\documentclass[10pt]{amsart}
\usepackage{amssymb, enumitem}
\usepackage[all]{xy}
\usepackage{hyperref, aliascnt}
\usepackage{mathtools}
\usepackage{bbm}
\usepackage[english]{babel}
\usepackage{tikz}
\def\today{\number\day\space\ifcase\month\or   January\or February\or
   March\or April\or May\or June\or   July\or August\or September\or
   October\or November\or December\fi\   \number\year}

\theoremstyle{definition}
\newtheorem{lma}{Lemma}[section]

\newaliascnt{thmCt}{lma}
\newtheorem{thm}[thmCt]{Theorem}
\aliascntresetthe{thmCt}

\newaliascnt{corCt}{lma}
\newtheorem{cor}[corCt]{Corollary}
\aliascntresetthe{corCt}

\newaliascnt{cnjCt}{lma}

\aliascntresetthe{cnjCt}

\newaliascnt{propCt}{lma}
\newtheorem{prop}[propCt]{Proposition}
\aliascntresetthe{propCt}

\newtheorem*{thm*}{Theorem}
\newtheorem*{cor*}{Corollary}
\newtheorem*{prop*}{Proposition}

\newcounter{theoremintro}

\newaliascnt{pgrCt}{lma}

\aliascntresetthe{pgrCt}

\newaliascnt{dfCt}{lma}
\newtheorem{df}[dfCt]{Definition}
\aliascntresetthe{dfCt}

\newaliascnt{remCt}{lma}
\newtheorem{rem}[remCt]{Remark}
\aliascntresetthe{remCt}

\newaliascnt{remsCt}{lma}

\aliascntresetthe{remsCt}

\newaliascnt{egCt}{lma}
\newtheorem{eg}[egCt]{Example}
\aliascntresetthe{egCt}

\newaliascnt{egsCt}{lma}

\aliascntresetthe{egsCt}

\newaliascnt{exCt}{lma}

\aliascntresetthe{exCt}

\newaliascnt{qstCt}{lma}
\newtheorem{qst}[qstCt]{Question}
\aliascntresetthe{qstCt}

\newaliascnt{pbmCt}{lma}
\newtheorem{pbm}[pbmCt]{Problem}
\aliascntresetthe{pbmCt}

\newaliascnt{notaCt}{lma}
\newtheorem{nota}[notaCt]{Notation}
\aliascntresetthe{notaCt}

\newcommand{\beq}{\begin{equation}}
\newcommand{\eeq}{\end{equation}}
\newcommand{\beqa}{\begin{eqnarray*}}
\newcommand{\eeqa}{\end{eqnarray*}}
\newcommand{\bal}{\begin{align*}}
\newcommand{\eal}{\end{align*}}
\newcommand{\bi}{\begin{itemize}}
\newcommand{\ei}{\end{itemize}}
\newcommand{\be}{\begin{enumerate}}
\newcommand{\ee}{\end{enumerate}}

\newcommand{\ep}{\varepsilon}

\newcommand{\Z}{{\mathbb{Z}}}
\newcommand{\C}{{\mathbb{C}}}
\newcommand{\N}{{\mathbb{N}}}

\newcommand{\Hi}{{\mathcal{H}}}

\newcommand{\B}{{\mathcal{B}}}
\newcommand{\U}{{\mathcal{U}}}

\newcommand{\T}{{\mathbb{T}}}

\newcommand{\On}{{\mathcal{O}_n}}
\newcommand{\Onp}{{\mathcal{O}_n^p}}
\newcommand{\Ot}{{\mathcal{O}_2}}
\newcommand{\Otp}{{\mathcal{O}_2^p}}

\newcommand{\id}{{\mathrm{id}}}

\newcommand{\supp}{{\mathrm{supp}}}

\newcommand{\Aut}{{\mathrm{Aut}}}

\newcommand{\ca}{$C^*$-algebra}

\newcommand{\I}{\infty}

\title[]{A modern look at algebras of operators on \texorpdfstring{$L^{\MakeLowercase{p}}$}{Lp}-spaces}

\date{\today}

\thanks{This research was supported
by a Postdoctoral Research Fellowship
from the Humboldt Foundation, by the
Deutsche Forschungsgemeinschaft (SFB 878), and
by a DFG's eigene Stelle.}

\author[Eusebio Gardella]{Eusebio Gardella}
\address{Eusebio Gardella
Mathematisches Institut, Fachbereich Mathematik und Informatik der
Universit\"at M\"unster, Einsteinstrasse 62, 48149 M\"unster, Germany.}
\email{gardella@uni-muenster.de}
\urladdr{www.math.uni-muenster.de/u/gardella/}

\begin{document}

\begin{abstract}
The study of operator algebras on Hilbert spaces, and \ca s in 
particular, is one of the most active areas within Functional 
Analysis. A natural generalization of these is to replace Hilbert 
spaces (which are $L^2$-spaces) with 
$L^p$-spaces, for $p\in [1,\I)$. The study of such algebras 
of operators is notoriously more challenging, due to the very 
complicated geometry of 
$L^p$-spaces by comparison with Hilbert spaces.

We give a modern overview of a research area 
whose beginnings can be traced back to the 50's, and that has seen 
renewed attention in the last decade through the infusion of new 
techniques. The combination of these new ideas with old tools was the
key to answer some long standing questions. Among others, we
provide a description of
all unital contractive homomorphisms between algebras of 
$p$-pseudofunctions of groups. 
\end{abstract}

\maketitle

\tableofcontents

\renewcommand*{\thetheoremintro}{\Alph{theoremintro}}
\section{Introduction}

Given $p\in [1,\infty)$, we say that a Banach algebra $A$ is an \emph{$L^p$-operator algebra}
if it admits an isometric representation on an $L^p$-space. $L^p$-operator algebras have been
historically studied by example, starting with Herz's influential works \cite{Her_theory_1971}
on harmonic analysis on $L^p$-spaces. Given a locally compact group $G$, Herz studied
the Banach algebra $PF_p(G)\subseteq \B(L^p(G))$ generated by the left regular representation,
as well as its weak-$\ast$ closure $PM_p(G)$ and its double commutant $CV_p(G)$. The study of
the structure of these algebras has attracted the attention of a number of mathematicians in
the last decades (see, for example, \cite{Cow_predual_1998}, \cite{NeuRun_column_2009},
\cite{Der_book_2011}, \cite{Der_property_2009}, and \cite{DerFilMon_ideal_2004}), 
particularly in what refers to the
so-called ``convolvers and pseudomeasures'' problem, which asks whether $CV_p(G)=PM_p(G)$ for all
groups $G$ and for all $p\in [1,\I)$. We refer the interested reader to the recent paper \cite{DawSpr_convoluters_2019} 
for an excellent
survey on the problem as well as for a proof that $CV_p(G)=PM_p(G)$ when $G$ has the approximation
property.

$L^p$-operator algebras have recently seen renewed interest, thanks to the infusion of ideas and
techniques from operator algebras, particularly in the works of
Phillips \cite{Phi_analogs_2012, Phi_crossed_2013}. There, Phillips introduced and studied
the $L^p$-analogs $\Onp$ of the Cuntz algebras $\mathcal{O}_n$ from \cite{Cun_simple_1977}
(which are the case $p=2$), and of UHF-algebras.
The work of Phillips motivated other authors to study $L^p$-analogs of well-studied families
of \ca s. These classes include group algebras \cite{Phi_crossed_2013, GarThi_group_2015,
GarThi_representations_2019}; groupoid algebras \cite{GarLup_representations_2017};
crossed products by topological systems \cite{Phi_crossed_2013}; AF-algebras \cite{PhiVio_classification_2017,
GarLup_nonclassifiability_2016}; and graph algebras \cite{CorRod_operator_2017}. In these works,
an $L^p$-operator algebra is obtained from combinatorial or dynamical data, and properties
of the underlying data are related to properties of the algebra. 
Quite surprisingly, the lack of symmetry of the unit ball
of an $L^p$-space for $p\neq 2$ allows one to prove isomorphism results that show a stark
contrast with the case $p=2$.

More recent works have approached the study of $L^p$-operator algebras in a more
abstract and systematic way \cite{GarThi_banach_2015, GarThi_extending_2017, BlePhi_operator_2018}, 
showing that there is an interesting theory waiting to be unveiled, of which only very little is 
currently known.

These notes are an introduction to $L^p$-operator algebras, beginning in Section~2 with what is arguably 
the most fundamental result in the area: Lamperti's theorem (see \autoref{thm:Lamperti}), 
which characterizes the invertible isometries of an $L^p$-space for $p\in [1,\I)\setminus\{2\}$.
In Section~3 we define $L^p$-operator algebras, prove some elementary facts about them, and give 
some basic examples. The next four sections are devoted to the study of three very prominent 
classes of examples: group algebras (Sections~4 and~5); Cuntz and graph algebras (Section~6); and 
crossed products (Section~7). Finally, in Section~8 we discuss a recent result obtained in 
\cite{ChoGarThi_rigidity_2019}: $\mathcal{O}_2^p\otimes \mathcal{O}_2^p$ is not isomorphic to
$\mathcal{O}_2^p$ for $p\in [1,\I)\setminus\{2\}$ (while it is well-known that an isomorphism 
exists for $p=2$; see \cite{Ror_short_1994}). This answers a question of Phillips. 

Most of this work is expository, although the exposition given here 
is quite different from what has appeared elswehere. 
Throughout the document there are several results that have not appeared
in the literature before, we give new proofs of some known results, 
and the content of Section~5 is mostly new.
\newline

\textbf{Acknowledgements:} This manuscript grew out of notes from a 
course given at the \emph{Instituto de 
Matem\'atica y Estad\'istica Rafael Laguardia} of the \emph{Facultad de Ingenier\'ia, 
Universidad de la Rep\'ublica} in Montevideo, Uruguay. 
The author would like to thank all the participants of the course for their valuable
feedback and the stimulating learning atmosphere. 

\section{Lamperti's theorem}

In \cite{Lam_isometries_1958}, 
Lamperti gave a description of the linear isometries of the $L^p$-space of a $\sigma$-finite measure space, for $p\in[1,\infty)$ with $p\neq 2$.
This result had been earlier announced (without proof) by Banach 
for the unit interval with
the Lebesgue measure, and for this reason it is also sometimes referred to
as the ``Banach-Lamperti Theorem''.
In this section, which is based on Sections~2 and~3 of \cite{GarThi_isomorphisms_2018}, we generalize Lamperti's result by characterizing the surjective, 
linear isometries on the $L^p$-space of a localizable measure algebra; see \autoref{thm:Lamperti}.
The generalization from $\sigma$-finite spaces to localizable ones 
will allow us in the next sections to deal with locally compact groups
that are not $\sigma$-compact. 
%Another reason to revisit Lamperti's theorem is that
%his original proof from \cite{Lam_isometries_1958} is not entirely correct as presented. 

\begin{df}\label{df:BooAlg}
A \emph{Boolean algebra} is a set $\mathcal{A}$ containing two 
distinguished elements $\emptyset$ and $I$, and with commutative, associative 
operations $\vee$ (disjoint union/orthogonal sum) and $\wedge$ (intersection/multiplication), 
and a notion of complementation $E\mapsto E^c$, satisfying the following properties:
\be
\item Idempotency: $E \vee E= E\wedge E=E$ for all $E\in \mathcal{A}$;
\item Absorption: $E \vee (E\wedge F)= E\wedge (E\wedge F)=E$ for all $E,F\in \mathcal{A}$;
\item Universality: for all $E\in\mathcal{A}$, we have 
\[E\vee \emptyset = E = E\wedge I, \ \ E\wedge \emptyset =\emptyset \ \ \mbox{ and }  
 \ \ E\vee I=I;\]
\item Complementation: $E\vee E^c=I$ and $E\wedge E^c=\emptyset$ for all $E\in\mathcal{A}$.
\ee
A \emph{homomorphism} between Boolean algebras is a function preserving all
the operations and the distinguished sets $\emptyset$ and $I$.

Given $E,F\in \mathcal{A}$, we write $E\leq F$ if $E\wedge F=E$, and we write 
$E\perp F$ if $E\wedge F=\emptyset$. 

We say that $\mathcal{A}$ is \emph{($\sigma$)-complete} 
if every nonempty countable subset of $\mathcal{A}$ has a supremum, 
and every
nonempty (countable) subset of $\mathcal{A}$ has an infimum.
\end{df}

The reader is referred to \cite{Fre_measure_2004} for a thorough treatment of Boolean algebras.
The most important example for the purposes of these notes is the following.

\begin{eg}\label{eg:QuotSigmaAlg}
Let $(X,\Sigma,\mu)$ be a measure space. Set 
$\mathcal{N}=\{E\in \Sigma\colon \mu(E)=0\}$, and let $\mathcal{A}$ denote the 
quotient $\Sigma/\mathcal{N}$. Then $\mathcal{A}$ is a $\sigma$-complete Boolean
algebra, with countable suprema given by union, and countable infima given by
intersection.
\end{eg}

There is a natural notion of measure on a Boolean algebra.

\begin{df}
Let $\mathcal{A}$ be a Boolean algebra. A map $\mu\colon \mathcal{A}\to [0,\I]$
is said to be a \emph{measure} if it satisfies $\mu(\emptyset)=0$ 
and $\mu(\bigvee_{n\in\N} E_n)=\sum_{n\in\N}\mu(E_n)$ whenever the $E_n$ are pairwise orthogonal.
We call a measure $\mu$ 
%\emph{strictly positive} if $\mu(E)>0$ whenever $E\neq \emptyse$, and 
\emph{semi-finite} if for every $E\in\mathcal{A}$ there exists $F\leq E$
with $0<\mu(F)<\I$.

Finally, we say that the measured algebra $(\mathcal{A},\mu)$ is \emph{localizable}
if $\mathcal{A}$ is $\sigma$-complete and $\mu$ is semi-finite.
\end{df}

\begin{eg}
In \autoref{eg:QuotSigmaAlg}, the map 
$\mathcal{A}\to [0,\I]$ which $\mu$ naturally induces is a measure. 
Localizability can be easily
characterized in terms of the measure space $(X,\Sigma,\mu)$: for every
$E\in\Sigma$ with $0<\mu(E)$, there exists $F\in \Sigma$ with $F\subseteq E$
such that $0<\mu(F)<\I$. 
\end{eg}

Our next goal is to define $L^p$-spaces associated to a measured algebra.
We denote by $\B(\mathbb{R})$ the Boolean algebra of all Borel-measurable
subsets of $\mathbb{R}$.

\begin{df}
Let $\mathcal{A}$ be a Boolean algebra. A \emph{measurable real-valued function}
is a Boolean homomorphism $f\colon \B(\mathbb{R})\to \mathcal{A}$ which preserves
suprema of countable sets. For $t\in \mathbb{R}$, we write 
$\{f>t\}$ for the set $f((t,\I))$.
\end{df}

Note that two functions $f,g\colon \B(\mathbb{R})\to \mathcal{A}$ are equal
if and only if $\{f>t\}=\{g>t\}$ for all $t\in\mathbb{R}$. 

\begin{eg}
In the context of \autoref{eg:QuotSigmaAlg}, a measurable function 
$f\colon X\to\mathbb{R}$ is identified with the homomorphism
$\widetilde{f}\colon \B(\mathbb{R})\to \mathcal{A}$ given by 
$\widetilde{f}(E)=f^{-1}(E)+\mathcal{N}\in \mathcal{A}$.
\end{eg}

The set of all measurable functions on $\mathcal{A}$ is denoted 
$L^0_{\mathbb{R}}(\mathcal{A})$. We set 
$L^0(\mathcal{A})=L^0_{\mathbb{R}}(\mathcal{A})+iL^0_{\mathbb{R}}(\mathcal{A})$.
For $f\in L_{\mathbb{R}}^0(\mathcal{A})$, we define its integral by
\[\int f\ d\mu= \int_0^\I \mu(\{f>t\}) \ dt.\]
For $f\in L^0(\mathcal{A})$ and $p\in [1,\I)$, we set
$\|f\|_p^p=\int |f|^{1/p} d\mu$.

The context of measured Boolean algebras seems to be the most appropriate one to 
do measure theory. The notion of $\sigma$-finiteness for measure spaces is 
technically very useful, but virtually every result for $\sigma$-finite 
measure spaces can be proved in the more general context of localizable 
measures. One instance of this is the Radon-Nikodym theorem; in fact, 
localizability is \emph{characterized} by the validity of the Radon-Nikodym theorem.

\begin{thm}\label{thm:RadonNikodym}
Let $\mathcal{A}$ be a $\sigma$-complete Boolean algebra, and let $\mu$ and 
$\nu$ be measures on $\mathcal{A}$ with $\mu$ localizable. 
Then there exists a unique function 
$\frac{d\nu}{d\mu}\in L^0_{\mathbb{R}}(\mathcal{A})$, called the \emph{Radon-Nikodym
derivative} of $\nu$ with respect to $\mu$, satisfying 
\[\int f\ d\nu = \int f\frac{d\nu}{d\mu} \ d\mu\]
for all $f\in L^1(\nu)$.
\end{thm}
\begin{proof}
We just give a rough idea of the proof.
Given $t\in (0,\I)$, the set
\[\{E\in\mathcal{A}\colon \nu(E)>t\mu(E)\}\]
has a largest element, which we denote by $D_t$. For $t\leq 0$, 
we set $D_t=\emptyset$.
One shows that the 
map $d\colon \mathbb{R}\to \mathcal{A}$ given by $d(t)=D_t$ for 
$t\in\mathbb{R}$, is order-preserving and 
order-continuous. Hence, there exists a function 
$\frac{d\nu}{d\mu}$ satisfying 
$\{\frac{d\nu}{d\mu}>t\}=D_t$ for all $t\in \mathbb{R}$. The 
identity in the statement is first verified for characteristic 
functions, and extended to $L^1(\nu)$ by density. 
\end{proof}

\begin{rem}\label{ex:ChangeVariables}
Let $(\mathcal{A},\mu)$ be a localizable measured Boolean algebra, and let 
$\varphi\in\Aut(\mathcal{A})$. 
Then
\[\int f\ d\mu = \int (\varphi\circ f) \frac{d(\mu\circ\varphi^{-1})}{d\mu} \ d\mu\]
for all $f\in L^0_{\mathbb{R}}(\mathcal{A})$. This identity is known as the ``change
of variables formula''. 
\end{rem}

The problem we will address in the rest of this section is to describe all
isometric isomorphisms (surjective isometries) between $L^p$-spaces, for
$p\in [1,\I)$. For $\ell^p(\{0,1\})$, this is easy to answer.

\begin{eg}
Endow $\{0,1\}$ with the counting measure. Below we give pictures of the 
unit balls of $\ell^p(\{0,1\})$, for $p=1,\I$ (in the real case). 

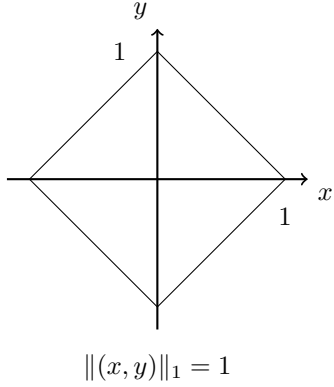
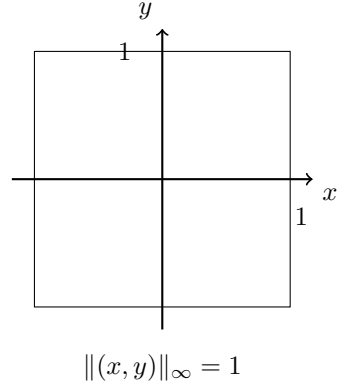
\begin{tikzpicture}
\draw[thick,->] (-2,0) -- (2,0) node[anchor=north west] {$x$};
\draw[thick,->] (0,-2) -- (0,2) node[anchor=south east] {$y$}; 
\draw (0,1.7) -- (1.7,0) -- (0,-1.7) -- (-1.7,0) -- (0,1.7);
\draw (0,-2.5) node {$\|(x,y)\|_1=1$};
\draw (-0.5,1.7) node {1};
\draw (1.7,-0.5) node {1};

\draw[thick,->] (4.5,0) -- (8.5,0) node[anchor=north west] {$x$};
\draw[thick,->] (6.5,-2) -- (6.5,2) node[anchor=south east] {$y$}; 
\draw (4.8,1.7) -- (8.2,1.7) -- (8.2,-1.7) -- (4.8,-1.7) -- (4.8,1.7);
\draw (6.5,-2.5) node {$\|(x,y)\|_\I=1$};
\draw (6,1.7) node {1};
\draw (8.35,-0.5) node {1};
\end{tikzpicture}

As $p$ goes from 1 to $\I$, the unit balls of the respective $\ell^p$
spaces change their shapes, going from the rhombus ($p=1$) to the square ($p=\I$), 
the case $p=2$ being a circle.\footnote{One 
should imagine the sides of the unit ball of $\ell^1(\{0,1]\})$ 
being ``blown out'' as $p$ grows.} 
The geometric description
of these unit balls reveals that the case $p=2$ has \emph{many}
more symmetries than the other ones. In particular, for $p\neq 2$, it is clear
that $\delta_0$ must be mapped either to a complex multiple of $\delta_0$ or 
to a complex multiple of $\delta_1$, and similarly for $\delta_1$. 
In other words, an invertible isometry in this
case has one of the following forms:
\[\begin{bmatrix}
    \lambda_1 & 0 \\
    0 & \lambda_2
  \end{bmatrix} \ \ \mbox{ or } \ \ \begin{bmatrix}
    0 & \lambda_1 \\
    \lambda_2 & 0
  \end{bmatrix}
\]
for $\lambda_1,\lambda_2\in S^1$. For $p=2$, rotations by angles other than
multiples of $\pi/2$ also give rise to invertible isometries 
(also known as unitary matrices), which are not isometric when regarded as 
maps on $\ell^p(\{0,1\})$ for $p\neq 2$.
\end{eg}

\begin{prop}\label{prop:IsometriesLp}
Let $(\mathcal{A},\mu)$ be a localizable measured Boolean algebra and let $p\in [1,\I)$.
\be
\item Set
\[\mathcal{U}(L^\I(\mu))=\{f\in L^0(\mathcal{A})\colon \{|f|>1\}=\{|f|<1\}=\emptyset\},\]
which is a group under multiplication. Then there is a group homomorphism
\[m\colon \U(L^\I(\mu))\to \mathrm{Isom}(L^p(\mu))\]
given by $m_f(\xi)=f\xi$ for all $f\in \U(L^\I(\mu))$ and all $\xi\in L^p(\mu)$.
\item There is a group homomorphism 
\[u\colon \Aut(\mathcal{A})\to \mathrm{Isom}(L^p(\mu))\]
given by
\[u_\varphi(\xi)=\varphi\circ\xi \left(\frac{d(\mu\circ\varphi^{-1})}{d\mu}\right)^{1/p}\]
for all $\varphi\in\Aut(\mathcal{A})$ and all $\xi\in L^p(\mu)$. 
\item Let $\varphi\in\Aut(\mathcal{A})$ and let $f\in \U(L^\I(\mu))$. Then
\[u_\varphi m_f u_\varphi^{-1}=m_{\varphi\circ f}.\]
In particular, there exists a group homomorphism 
\[\U(L^\I(\mu))\rtimes \Aut(\mathcal{A})\to \mathrm{Isom}(L^p(\mu)).\]
\item Given $f,g\in \U(L^\I(\mu))$ and $\varphi,\psi\in \Aut(\mathcal{A})$, we 
have 
\[\|m_fu_\varphi - m_gu_\psi\|=\max \{\|f-g\|_\I, 2-2\delta_{\varphi,\psi}\}.\]
\ee
\end{prop}
\begin{proof}
Parts~(1), (2) and (3) are routine. For example, to prove in~(2) that 
$u_\varphi$ is isometric, we let $\xi\in L^p(\mu)$ be given and use \autoref{ex:ChangeVariables} at 
the third step to get
\begin{align*}
\|u_\varphi(\xi)\|_p^p&= \int\left|(\varphi\circ\xi)\left(\frac{d(\mu\circ\varphi^{-1})}{d\mu}\right)^{1/p}\right|^{p}d\mu\\
&=\int|\varphi\circ\xi|^p\frac{d(\mu\circ\varphi^{-1})}{d\mu}\ d\mu\\
&= \int |\xi|^p d\mu\\
&= \|\xi\|_p^p.
\end{align*}

We prove (4). 
Fix $f,g\in\mathcal{U}(L^\infty(\mu))$ and 
$\varphi,\psi\in\Aut(\mathcal{A})$.

Suppose that $\varphi\neq\psi$; we will show that
$\| m_f u_\varphi - m_g u_\psi \|=2$.
Choose a $E\in\mathcal{A}$ with $\varphi(E)\neq\psi(E)$.
If $\varphi(E)\setminus\psi(E)\neq \emptyset$, 
set $F=\varphi^{-1}(\varphi(E)\setminus\psi(E))$.
Otherwise, if $\psi(E)\setminus\varphi(E)\neq \emptyset$, set $F=\psi^{-1}(\psi(E)\setminus\varphi(E))$.
In either case, we have $F\neq \emptyset$ and $\varphi(F)\wedge\psi(F)=0$.
Using that $\mu$ is localizable, choose $F_0\in\mathcal{A}$ with 
$0\neq F_0\leq F$ and $\mu(F_0)<\infty$.
Set $\xi=\tfrac{1}{\mu(F_0)^{1/p}}\mathbbm{1}_{F_0}$, which belongs to $L^p(\mu)$ and satisfies $\|\xi\|_p=1$.
By construction, $(m_f u_\varphi)(\xi)$ and $(m_g u_\psi)(\xi)$ have disjoint supports, whence
\[
\| m_f u_\varphi - m_g u_\psi \|
\geq \| ( m_f u_\varphi - m_g u_\psi)(\xi) \|_p
= \| (m_f u_\varphi)(\xi) \|_p + \|(m_g u_\psi)(\xi) \|_p = 2.
\]
On the other hand, we have $\| m_f u_\varphi - m_g u_\psi \|\leq \| m_f u_\varphi\| + \| m_g u_\psi \|=2$, as desired.

We now assume that $\varphi=\psi$.
Then $\|m_f u_\varphi - m_g u_\psi \|= \|m_f-m_g\|$.
Moreover, $m_f-m_g=m_{f-g}$, and it is straightforward to 
verify that this operator has norm
$\|f-g\|_\infty$, which is at most 2. This finishes the proof.
\end{proof}

%\begin{ex}
%Complete the proof of \autoref{prop:IsometriesLp}.
%\end{ex}

The main result of this section, \autoref{thm:Lamperti}, asserts that 
for $p\neq 2$, the only isometries of $L^p(\mu)$ for localizable $\mu$ are 
the ones described in \autoref{prop:IsometriesLp}. In other words,
the homomorphism from part~(3) of \autoref{prop:IsometriesLp} is 
an isomorphism.

We need a preparatory lemma.

\begin{lma}\label{lma:Clarkson}
Let $(\mathcal{A},\mu)$ be a measured Boolean algebra, let $p\in [1,\I)$,
and let $\xi,\eta\in L^p(\mu)$.
\be
\item For $2\leq p$, we have
\[\|\xi+\eta\|_p^p+ \|\xi-\eta\|_p^p \geq 2\left(\|\xi\|_p^p+\|\eta\|_p^p\right),\]
and equality holds for $p\neq 2$ if and only if $\xi\eta=0$.
\item For $p\leq 2$, we have
\[\|\xi+\eta\|_p^p+ \|\xi-\eta\|_p^p \leq 2\left(\|\xi\|_p^p+\|\eta\|_p^p\right),\]
and equality holds for $p\neq 2$ if and only if $\xi\eta=0$.
\item If $p\neq 2$ and $T\colon L^p(\mu)\to L^p(\mu)$ is isometric, then
$T(\xi)T(\eta)=0$ whenever $\xi\eta=0$. (In other words, $T$ is ``disjointness preserving''.)
\ee
\end{lma}
\begin{proof}
(1). Set $\phi(t)=t^p$ for $t\in \mathbb{R}$.
Then $\phi(\sqrt(t))$ is convex, and by standard results, we have
\[\phi^{-1}\left(\frac{\phi(|z+w|)+\phi(|z-w|)}{2}\right)
 \geq (|z|^2+|w|^2)^{1/2}\geq \phi^{-1}(\phi(|z|)+\phi(|w|)),
\]
for all $z,w\in \C$. Moreover, when $\phi$ is strictly convex, equality holds
if and only if $zw=0$. Since $\phi^{-1}$ is increasing, the result follows by
integration.

(2). This is entirely analogous to (1).

(3). Suppose $p\neq 2$. If $\xi\eta=0$, then
\begin{align*}
\|T(\xi)+T(\eta)\|_p^p+\|T(\xi)-T(\eta)\|_p^p&=\|\xi+\eta\|_p^p+ \|\xi-\eta\|_p^p \\
&= 2\left(\|\xi\|_p^p+\|\eta\|_p^p\right)\\
&= 2\left(\|T(\xi)\|_p^p+\|T(\eta)\|_p^p\right). 
\end{align*}
It follows that $T(\xi)T(\eta)=0$, as desired.
\end{proof}

We have now arrived at Lamperti's theorem.

\begin{thm}\label{thm:Lamperti}
Let $(\mathcal{A},\mu)$ be a localizable measured Boolean algebra, 
let $p\in [1,\I)\setminus\{2\}$ and let $T\colon L^p(\mu)\to L^p(\mu)$
be an invertible isometry. Then there exist unique $\varphi\in\Aut(\mathcal{A})$
and $f\in \U(L^\I(\mu))$ such that $T=m_fu_\varphi$. 
\end{thm}
\begin{proof}
To make the proof more transparent, we will assume that $(\mathcal{A},\mu)$ 
arises as in \autoref{eg:QuotSigmaAlg} from a \emph{finite} measure
space $(X,\Sigma,\mu)$. Given $E\in \mathcal{A}=\Sigma/\mathcal{N}$, 
set
\[\varphi(E)=\{|T(\mathbbm{1}_E)|>0\}.\]
If $E,F\in\mathcal{A}$ are disjoint, then so are $\varphi(E)$ and 
$\varphi(F)$ by part~(3) of \autoref{lma:Clarkson}. In particular, 
$\varphi(X-E)=\varphi(X)-\varphi(E)$, so $\varphi$ preserves complements and
is thus a homomorphism from $\mathcal{A}$ to itself. Since $T$ is invertible, so 
is $\varphi$. 

Using that $\mu(X)<\I$, set $h=T(\mathbbm{1}_X)$. 

We claim that $T(\xi)=(\xi\circ\varphi)h$ for all $\xi\in L^p(\mu)$. 
To check this, 
suppose first that $\xi=\mathbbm{1}_E$ for some $E\in\mathcal{A}$. Note that
\[h=T(\mathbbm{1}_E)+T(\mathbbm{1}_{E^c}),\]
and that the supports of the functions on the right-hand side are disjoint.
In particular, $h$ agrees with $T(\mathbbm{1}_E)$ on the support of $T(\mathbbm{1}_E)$, which
is $\varphi(E)$. Thus,
\[T(\mathbbm{1}_E)=h\mathbbm{1}_{\varphi(E)}=h(\mathbbm{1}_E\circ \varphi).\]
This verifies the claim for indicator functions, and thus for step functions.
Since these are dense in $L^p(\mu)$, the claim follows.

It remains to identify $h$. For $E\in\mathcal{A}$, we have
\[\mu(E)=\|\mathbbm{1}_E\|_p^p=\|T(\mathbbm{1}_E)\|_p^p=\int |h|^p \mathbbm{1}_{\varphi(E)}^p d\mu=\int_{\varphi(E)} |h|^p d\mu.\]
On the other hand, 
\[\mu(E)=(\mu\circ\varphi^{-1})(\varphi(E))=\int_{\varphi(E)}\frac{d(\mu\circ\varphi^{-1})}{d\mu} \ d\mu.\]
It follows that $|h|^p=\frac{d(\mu\circ\varphi^{-1})}{d\mu}$, so there exists
$f\in \U(L^\I(\mu))$ such that 
$h=f\left(\frac{d(\mu\circ\varphi^{-1})}{d\mu}\right)^{1/p}$. This finishes the proof.
\end{proof}

\begin{rem}\label{rem:liftTransf}
Suppose that $(\mathcal{A},\mu)$ arises from a measure
space $(X,\Sigma,\mu)$ as in \autoref{eg:QuotSigmaAlg}.
In some cases, it is possible to ``lift'' the automorphism $\varphi$ to a
bi-measurable bijective transformation $T\colon X\to X$ satisfying 
\[\mu(E)=0 \Leftrightarrow \mu(T(E))=0 \Leftrightarrow \mu(T^{-1}(E))=0. 
\]
This is always the case, for example, when $\mu$ is an atomic measure, 
in which case one even has $\mu(T(E))=\mu(E)$ for all $E\in\Sigma$.
\end{rem}

\begin{cor}
Let $(\mathcal{A},\mu)$ be a localizable measure algebra, let $p,q\in [1,\I)$
with $p\neq q$, and let $T\colon L^p(\mu)\cap L^q(\mu)\to L^p(\mu)\cap L^q(\mu)$
be a linear map that extends to isometric surjections $L^p(\mu)\to L^p(\mu)$ and 
$L^q(\mu)\to L^q(\mu)$. Then there exist $h\in \U(L^\I(\mu))$ and 
$\varphi\in\Aut(\mathcal{A})$ with $\mu\circ\varphi=\mu$ and 
\[T(\xi)=h(\varphi\circ\xi)\]
for all $\xi\in L^p(\mu)$.
\end{cor}
\begin{proof}
Without loss of generality, we may assume that $p\neq 2$. 
By \autoref{thm:Lamperti}, there exist $h\in \mathcal{U}(L^\I(\mu))$
and $\varphi\in\mathrm{Aut}(\mathcal{A})$ such that
\[T(\xi)=h(\varphi\circ\xi)\left(\frac{d(\mu\circ\varphi^{-1})}{d\mu}\right)^{1/p}\]
for all $\xi\in L^p(\mu)$. 
It is easily checked that $\|T(\eta)\|_q=\|\eta\|_q$ for all 
$\eta\in L^p(\mu)\cap L^q(\mu)$ if and only 
if $\frac{d(\mu\circ\varphi^{-1})}{d\mu}=1$, which is equivalent to 
$\mu\circ\varphi=\mu$. This finishes the proof.
\end{proof}

\section{\texorpdfstring{$L^p$}{Lp}-operator algebras: basic examples}
%Throughout this section, we fix $p\in [1,\I)$. We say that a Banach space $E$ is
%an $L^p$-space if there exists a measured Boolean algebra $(\mathcal{A},\mu)$
%such that $E$ is isometrically isomorphic to $L^p(\mu)$. 
%Every $L^p$-space can be realized by a localizable measure algebra, so that Lamperti's
%theorem (\autoref{thm:Lamperti}) applies.
Recall that a \emph{Banach algebra} is a complex algebra $A$ with a Banach space
structure, satisfying $\|ab\|\leq \|a\|\|b\|$ for all $a,b\in A$. When $A$ has a
multiplicative unit,
we moreover demand that $\|1\|=1$.

\begin{df}
Let $A$ be a Banach algebra. We say that $A$ is an \emph{$L^p$-operator algebra}
if there exist an $L^p$-space $E$ and an isometric homomorphism
$A\to \B(E)$. 

A \emph{representation} of $A$ (on an $L^p$-space $E$) is a contractive
homomorphism $\varphi\colon A\to \B(E)$. We say that $\varphi$ is \emph{non-degenerate}
if $\mathrm{span}\{\varphi(a)\xi\colon a\in A,\xi\in E\}$ is dense in $E$.
\end{df}

We make some comments on why we restrict to $p\in [1,\I)$. First, for $p<1$, the 
vector space $L^p(\mu)$ is not normed (and its dual space is in fact trivial). On the other hand,
for $p=\I$ we do not have a Lamperti-type theorem that allows us to represent the invertible
isometries spatially (neither do we for $p=2$, but in this case there are adjoints). For 
$p\in (1,\I)$, the fact that an $L^p$-space is reflexive is sometimes quite useful and produces
legitimate differences with the case $p=1$; see, for example, \autoref{thm:GpAmen}.

\begin{rem} 
For $p=2$, an $L^2$-operator algebra is a not necessarily self-adjoint operator algebra. 
\end{rem}

\begin{eg}\label{eg:Mn}
If $E$ is an $L^p$-space, then $\B(E)$ is trivially an $L^p$-operator algebra. 
When $E=\ell^p(\{1,\ldots,n\})$, then $\B(E)$ is algebraically isomorphic to $M_n$,
and we denote the resulting Banach algebra by $M_n^p$. \end{eg}

The norm described above is not the only $L^p$-operator 
norm on $M_n$ (even for $p=2$): 

\begin{eg} If $s\in M_n^p$ is an invertible
operator, one can define a new $L^p$-operator norm $\|\cdot\|_s$ on $M_n$ by setting $\|x\|_s=\|sxs^{-1}\|$. This norm is in general different from the 
one on $M_n^p$.
\end{eg}

\begin{eg}
Let $X$ be a locally compact topological space. Then $C_0(X)$ is an $L^p$-operator
algebra. In the case that there exists a regular Borel measure $\mu$ on $X$, one 
can represent $C_0(X)$ isometrically on $L^p(\mu)$ via multiplication operators.
(Such a measure does not always exist, but one can always find a ``separating'' 
family of such measures and take the direct sum of the resulting
representations by multiplication.)
\end{eg}

It is very convenient to work with non-degenerate representations. 
However, such representations don't always exist.

\begin{eg}
Let $\C_0$ be the Banach algebra whose underlying Banach space is $\C$, endowed
with the trivial product: $ab=0$ for all $a,b\in\C$. Then $\C_0$ is an $L^p$-operator
algebra for all $p\in [1,\I)$, since it is isometrically isomorphic to the upper-triangular matrices 
in $M_2^p$. However, $\C_0$ does not admit non-degenerate representations.
\end{eg}

For $L^p$-operator algebras with contractive approximate identities, one can always
``cut-down'' a given representation to obtain a non-degenerate one. This is much more
subtle than in the Hilbert space case (where one just co-restricts to the essential
range, which is automatically a Hilbert space), 
since there are subspaces of an $L^p$-space 
which are not themselves $L^p$-spaces. The case of unital algebras is much easier 
to prove.

\begin{prop}
Let $A$ be a unital $L^p$-operator algebra. Then $A$ admits an isometric, unital
representation $\varphi\colon A\to \B(E)$ on an $L^p$-space $E$.  
\end{prop}
\begin{proof}
Let $F$ be an $L^p$-space and let $\psi\colon A\to \B(F)$ be an isometric 
representation. Set $e=\varphi(1)$, which is a contractive idempotent. By the main
result of \cite{Tza_remarks_1969}, the image $E=e(F)$ of $e$ is an $L^p$-space.
Define $\varphi\colon A\to \B(E)$ by $\varphi(a)(\xi)=\psi(a)(\xi)$ for all $a\in A$
and all $\xi\in E\subseteq F$. Then $\varphi$ is an isometric representation.
\end{proof}

For general algebras with a contractive approximate identity, one shows that there
is a contractive idempotent from the ambient $L^p$-space to the essential range, 
which implies that the essential range is an $L^p$-space; see \cite{GarThi_extending_2017}.

The study of $L^p$-operator algebras is generally much more complicated than that
of ($L^2$-)operator algebras, largely due to the complicated geometry of $L^p$-spaces.
Even for algebras that ``look like'' \ca s, many of the most fundamental facts about
\ca s fail. To mention a few:
\bi
\item There is so far no abstract characterization of $L^p$-operator algebras among all
Banach algebras, or a canonical way of obtaining a representation on an $L^p$-space 
for a given $L^p$-operator algebra;
\item $L^p$-operator norms are not unique; in particular, a homomorphism between
$L^p$-operator algebras does not necessarily have closed range, and an injective
homomorphism is not necessarily isometric.
\item For $p\neq 2$, not every quotient of an $L^p$-operator algebra can be
represented on an $L^p$-space; see \cite{GarThi_quotients_2016}. 
There is also no known characterization of which ideals give 
$L^p$-operator quotients.
\ei

We do not yet have a general theory, and it has been very productive to study concrete
families of $L^p$-operator algebras, typically (but not always) 
constructed from some topological/algebraic data.
In the following sections, we will introduce some of the most studied classes of
examples.

\section{Group algebras acting on \texorpdfstring{$L^p$}{Lp}-spaces}
Let $G$ be a locally compact group, and let $\mu$ denote its Haar measure. 
For functions $f$ and $g$ defined on $G$, their \emph{convolution} is defined as
\[(f\ast g)(s)=\int_G f(t)g(t^{-1}s)\ d\mu(t)\]
for all $s\in G$, whenever the integral is finite.
When $f\in L^1(G)$ and $g\in L^p(G)$, the convolution $f\ast g$
belongs to $L^p(G)$, and $\|f\ast g\|_p\leq \|f\|_1\|g\|_p$. 
In particular, $L^1(G)$ is a Banach algebra under convolution, and it is 
unital if and only if $G$ is discrete. We denote by
$\lambda_p\colon L^1(G)\to \B(L^p(G))$ the action by left convolution (also called
the \emph{left regular representation}). 

It is a standard fact in abstract harmonic analysis that unital, 
contractive representations of $L^1(G)$ on arbitrary Banach spaces
are in one-to-one correspondence with isometric representations of $G$,
via the integrated form. The case of discrete
groups is particularly easy to prove:

\begin{prop}\label{ex:ReprL1G}
Let $G$ be a discrete group, and let $E$ be any Banach space. 
Given a unital, contractive homomorphism $\varphi\colon \ell^1(G)\to \B(E)$, 
let $u_\varphi\colon G\to \mathrm{Isom}(E)$ be given by $u_\varphi(g)=\varphi(\delta_g)$ for $g\in G$. Conversely, given $u\colon G\to\mathrm{Isom}(E)$,
let $\varphi_u\colon \ell^1(G)\to \B(E)$ be the bounded linear map 
determined by $\varphi_u(\delta_g)=u_g$ for all $g\in G$.
\be\item The map $u_\varphi$ is an isometric representation of $G$ on $E$
(that is, $u_\varphi(g)$ 
is an invertible isometry of $E$ for all $g\in G$, and $u_\varphi(gh)=u_\varphi(g)u_\varphi(h)$ for all $g,h\in G$.)
\item The map $\varphi_u$ extends to a well-defined algebra homomorphism, which is unital and contractive.
\item We have $\varphi_{u_\varphi}=\varphi$ and $u_{\varphi_u}=u$. 
\ee
\end{prop}
\begin{proof}
(1). To show that $u_\varphi$ is a group homomorphism, let $g,h\in G$.
Then
\[u_\varphi(gh)=\varphi(\delta_{gh})=\varphi(\delta_g\ast\delta_h)=
 \varphi(\delta_g)\varphi(\delta_h)=u_\varphi(g)u_\varphi(h).
\]

Fix $g\in G$. We now show that $u_\varphi(g)$ is an invertible isometry.
It is clearly invertible, with inverse given by $u_\varphi(g^{-1})$.
On the other hand, since $\varphi$ is contractive, we have 
$\|u_\varphi(g)\|\leq 1$. Since this also applies to $g^{-1}$, we get
\[\|\xi\|=\|u_\varphi(g^{-1})(u_\varphi(g)(\xi))\|\leq \|u_\varphi(g)(\xi)\|\leq \|\xi\|\]
for all $\xi\in E$. It follows that $\|u_\varphi(g)(\xi)\|=\|\xi\|$ and 
thus $u_\varphi(g)$ is an isometry.

(2). Given $f\in c_c(G)\subseteq \ell^1(G)$, written as a finite linear
combination $f=\sum\limits_{g\in \supp(f)} a_g\delta_g$, we have
$\varphi_u(f)=\sum\limits_{g\in \supp(f)} a_gu_g$ and thus
\[\|\varphi_u(f)\|\leq \sum\limits_{g\in \supp(f)} |a_g|\|u_g\|= \sum\limits_{g\in \supp(f)} |a_g|=\|f\|_1.\]
Since $c_c(G)$ is dense in $\ell^1(G)$, it follows that $\varphi_u\colon c_c(G)\to \B(E)$ extends to a well-defined unital contractive representation $\varphi_u\colon \ell^1(G)\to \B(E)$.

(3). Given $f\in c_c(G)\subseteq \ell^1(G)$ written as a finite linear
combination $f=\sum\limits_{g\in \supp(f)} a_g\delta_g$, we have
\[\varphi_{u_\varphi}(f)=\sum\limits_{g\in \supp(f)} a_gu_\varphi(g)=
 \sum\limits_{g\in \supp(f)} a_g\varphi(\delta_g)=\varphi(f).
\]
Since $c_c(G)$ is dense in $\ell^1(G)$, we deduce that $\varphi_{u_\varphi}=\varphi$. Finally, given $g\in G$ we have
\[u_{\varphi_u}(g)=\varphi_u(\delta_g)=u_g,\]
so $u_{\varphi_u}=u$, as desired.
\end{proof}

\begin{eg}\label{eg:LeftReg}
In the context of the previous proposition, take 
$E=\ell^p(G)$, and let $\texttt{Lt}\colon G\to\ell^p(G)$ be
given by $\texttt{Lt}_g(\xi)(h)=\xi(g^{-1}h)$ 
for all $g,h\in G$ and all $\xi\in E$.
Then $\varphi_{\texttt{Lt}}=\lambda_p$.
\end{eg}
\begin{proof}
By part~(3) of \autoref{ex:ReprL1G}, it suffices to check that 
$u_{\lambda_p}=\texttt{Lt}$. Given $g\in G$ and $\xi\in \ell^p(G)$,
we have
\[u_{\lambda_p}(g)(\xi)=\lambda_p(\delta_g)(\xi)=\delta_g\ast\xi.\]
Hence, for $h\in G$ we have
\[u_{\lambda_p}(g)(\xi)(h)=(\delta_g\ast\xi)(h)=\int_G\delta_g(s)\xi(s^{-1}h)\ ds=\xi(g^{-1}h).\]
We conclude that $u_{\lambda_p}=\texttt{Lt}$, as desired.
\end{proof}

The following definitions are due to Herz \cite{Her_theory_1971}.

\begin{df}\label{df:GroupLpOpAlgs}
Let $G$ be a locally compact group.
\be
\item 
We define the \emph{algebra of $p$-pseudofunctions} $PF_p(G)$ of $G$, 
to be 
\[PF_p(G)=\overline{\lambda_p(L^1(G))}^{\|\cdot\|}\subseteq \B(L^p(G)).\]
\item 
We define the \emph{algebra of $p$-pseudomeasure} $PM_p(G)$ of $G$ 
to be 
\[PM_p(G)=\overline{\lambda_p(L^1(G))}^{\mathrm{WOT}}\subseteq \B(L^p(G)).\]
\item 
We define the \emph{algebra of $p$-convolvers} $CV_p(G)$ of $G$ 
to be 
\[CV_p(G)=\lambda_p(L^1(G))''\subseteq \B(L^p(G)).\]
\ee\end{df}

In general, we have $PF_p(G)\subseteq PM_p(G)\subseteq CV_p(G)$.

\begin{nota}
In these notes, we will usually denote the algebra of $p$-pseudo\-functions
on $G$, defined in (1)~above, by $F^p_\lambda(G)$. This algebra is also sometimes called the 
``reduced group $L^p$-operator algebra''.
\end{nota}

There is another very important 
$L^p$-operator algebra associated to a locally compact group, 
whose construction is due to
Phillips \cite{Phi_crossed_2013}.

\begin{df} \label{df:FullGroupLp}
Let $G$ be a locally compact group and let $p\in [1,\I)$. We define its 
\emph{full group $L^p$-operator algebra} $F^p(G)$ to be the completion of
$L^1(G)$ in the norm 
\[\|f\|_{F^p(G)}=\sup\{\|\varphi(f)\|\colon \varphi\colon L^1(G)\to \B(E) \mbox{ contractive
 representation}\},\]
for $f\in L^1(G)$. (Where $E$ ranges over all possible $L^p$-spaces.)
\end{df}

It is not entirely obvious from the definition that $F^p(G)$ is indeed an $L^p$-operator algebra
(unlike for the algebras defined in \autoref{df:GroupLpOpAlgs}, which are explicitly constructed 
as a Banach subalgebras of $\B(L^p(G))$).
For this, one needs to produce an isometric representation on some $L^p$-space.

\begin{prop}
Let $G$ be a locally compact group and let $p\in [1,\I)$. Then $F^p(G)$
is an $L^p$-operator algebra.
\end{prop}
\begin{proof}
For $f\in L^1(G)$ and $n\in\N$, let 
$\varphi_{f,n}\colon L^1(G)\to \B(E_{f,n})$ be a contractive representation
satisfying 
\[\|\varphi_{f,n}\|\geq \|f\|_{F^p(G)}-\frac{1}{n}.\]
Set $E=\bigoplus\limits_{f\in L^1(G)}\bigoplus\limits_{n\in\N} E_{f,n}$
and let $\varphi\colon L^1(G)\to \B(E)$ denote the ``diagonal'' representation.
Then $E$ is an $L^p$-space. Moreover, for $f\in L^1(G)$ and $n\in\N$, one has 
\[\|\varphi(f)\|=\sup_{m\in\N}\sup_{g\in L^1(G)}\|\varphi_{g,m}(f)\|\geq \|\varphi_{f,n}(f)\|\geq \|f\|_{F^p(G)}-\frac{1}{n}.\]
In particular, $\|\varphi(f)\|\geq \|f\|_{F^p(G)}$. Since $\varphi$ is a contractive
representation of $L^1(G)$ on some $L^p$-space, we also have 
$\|\varphi(f)\|\leq \|f\|_{F^p(G)}$, and hence 
$\|\varphi(f)\|= \|f\|_{F^p(G)}$. It follows that the norm-closure of 
$\varphi(L^1(G))$ in $\B(E)$ is isometrically isomorphic to $F^p(G)$, and thus 
$F^p(G)$ is an $L^p$-operator algebra.
\end{proof}

\begin{rem}
When $G$ is discrete, $F^p(G)$ is the universal $L^p$-operator algebra
generated by invertible isometries $u_s$, for $s\in G$, satisfying
$u_su_t=u_{st}$. 
\end{rem}

Since $\lambda_p\colon L^1(G)\to\B(L^p(G))$ is a contractive representation, it
follows that $\|\cdot\|_{F^p_\lambda(G)}\leq \|\cdot \|_{F^p(G)}$. In other words,
the identity map 
\[\id\colon (L^1(G),\|\cdot\|_{F^p_\lambda(G)})\to (L^1(G),\|\cdot\|_{F^p(G)})\]
is contractive, and it extends to a contractive map 
$\kappa_p\colon F^p(G)\to F^p_\lambda(G)$ between their completions, which 
has dense range since it contains $L^1(G)$. 

The cases $p=1$ and $p=2$ of the algebras in 
\autoref{df:GroupLpOpAlgs} and \autoref{df:FullGroupLp}
are easy to describe. For the identification of $F^2(G)$, we will need
the following easy fact:

\begin{lma}\label{ex:UnitariesBH}
Let $\Hi$ be a Hilbert space and let $u\in\B(\Hi)$. Show that $u$ is a 
unitary if and only if $u$ is invertible and $\|u\|=\|u^{-1}\|=1$.
\end{lma}
\begin{proof}
Let $\xi\in \Hi$. Then
\[\|\xi\|=\|u^{-1}(u(\xi))\|\leq \|u(\xi)\|\leq \|\xi\|,\]
since both $u$ and $u^{-1}$ are contractive. It follows that $\|u(\xi)\|=\|\xi\|$ and thus $u$ is an isometry. Since it is surjective, it follows
that it is a unitary.
\end{proof}

\begin{prop}\label{lma:p1p2}
Let $G$ be a locally compact group.
\bi\item When $p=1$, we get $F_\lambda^1(G)=F^1(G)=L^1(G)$ and $PM_1(G)=CV_1(G)=M(G)$.
\item When $p=2$, we get $F_\lambda^2(G)=C^*_\lambda(G)$, $F^2(G)=C^*(G)$ and $PM_2(G)=CV_2(G)=W^*(G)$.
\ei
\end{prop}
\begin{proof}
(1). Recall that $L^1(G)$ has a contractive approximate identity $(f_n)_{n\in\N}$.
Given $f\in L^1(G)$ and $\ep>0$, 
find $n\in\N$ such that $\|f\ast f_n\|_1\geq \|f\|_1-\ep$. Then
\[\|f\|_1\geq \|\lambda_1(f)\|_{\B(L^1(G))}\geq \frac{\|f\ast f_n\|_1}{\|f_n\|_1}\geq \|f\|_1-\ep.\]
Since $\ep>0$ is arbitrary, it follows that $\|\lambda_1(f)\|_{\B(L^1(G))}=\|f\|_1$
and thus $F_\lambda^1(G)=L^1(G)$. 

Since $\|\cdot\|_{F^1_\lambda(G)}\leq \|\cdot\|_{F^1(G)}\leq \|\cdot\|_{1}$, 
it follows that $\|\cdot\|_{F^1(G)}=\|\cdot\|_1$ and hence $F^1(G)=L^1(G)$ as well.
We omit the proofs of the identities $PM_1(G)=CV_1(G)=M(G)$, which are analogous.

(2). The identities $F_\lambda^2(G)=C^*_\lambda(G)$ and $PM_p(G)=W^*(G)$ are true by definition, and $CV_2(G)=W^*(G)$ by the double-commutant theorem. 
The identity $F^2(G)=C^*(G)$ follows from \autoref{ex:ReprL1G} and 
\autoref{ex:UnitariesBH}.
\end{proof}

In view of the previous proposition, we often regard the group algebras from
\autoref{df:GroupLpOpAlgs}, for different values of $p$, as a continuously 
varying family of Banach algebras that deform $L^1(G)$ or $M(G)$ into 
$C^*_\lambda(G)$, $C^*(G)$, or $W^*(G)$. 

The fact that $F_\lambda^1(G)$ and $F^1(G)$ agree is misleading, since, for other 
values of $p$, this happens if and only if $G$ is amenable. (Recall that a
group $G$ is \emph{amenable} if for every $\ep>0$ and for every compact subset
$K\subseteq G$, there exists a compact subset $F\subseteq G$ such that 
$\mu(FK\triangle F)<\ep \mu(F)$.) The following is Theorem~3.20 in~\cite{GarThi_group_2015},
and it was also independently proved by Phillips.

\begin{thm}\label{thm:GpAmen}
Let $G$ be a locally compact group and let $p>1$. Then the canonical map 
$\kappa_p\colon F^p(G)\to F^p_\lambda(G)$ is 
an isometric isomorphism if and only if $G$ is amenable. 
\end{thm}

This implies, among others, that for $G$ amenable the reduced group algebra
$F^p_\lambda(G)$ admits a characterization in terms of generators and relations.
For $G=\Z$, this description is particularly nice: 

\begin{cor}
$F_\lambda^p(\Z)$ is the universal $L^p$-operator algebra
generated by an invertible isometry and its inverse. 
For $p=2$, this algebra is isometrically
isomorphic to $C(S^1)$, but in general the norm is larger.
\end{cor}

A natural question that arises from looking at the cases $p=1,2$ is whether
the equality $PM_p(G)=CV_p(G)$ always holds. This is arguably the most important
open problem in the area, dating back to Herz's work in the 70's, 
and is known as the ``convolvers and pseudomeasures'' 
problem.

\begin{qst}\label{qst:ConvPseudoM}
Let $G$ be a locally compact group and let $p\in [1,\I)$. Is it true that  
$PM_p(G)=CV_p(G)$?
\end{qst}

The question above asks whether a specific case of the double-commutant theorem
holds for operators on $L^p$-spaces. It is known that \autoref{qst:ConvPseudoM}
has a positive answer, for all $p\in [1,\I)$, whenever $G$ has the so-called
approximation property; see \cite{DawSpr_convoluters_2019} for a 
particularly nice presentation of this result.
This is in particular the case when $G$ is amenable.

We finish this subsection by discussing the smallest non-trivial group algebra. 

\begin{eg}
$F_\lambda^p(\Z_2)$ is the Banach subalgebra of $\B(\ell^p(\{0,1\}))$ generated
by the rotation matrix $\begin{bmatrix}
    0 & 1 \\
    1 & 0
  \end{bmatrix}$. This algebra can be
identified with $\C^2$, but its norm is not the maximum norm. 
The norm of $(a,b)\in F_\lambda^p(\Z_2)$ is the $L^p$-operator norm of the matrix
$\frac{1}{2}\begin{bmatrix}
    a+b & a-b \\
    a-b & a+b
  \end{bmatrix}$.
\end{eg}

%\begin{ex} 
%Let $p\in [1,\I)$. Prove that
%\[\|(1,i)\|_{F_\lambda^p(\Z_2)}=2^{\left|\frac{1}{p}-\frac{1}{2}\right|}.\]
%Deduce that $F_\lambda^p(\Z_2)$ is not isometrically
%isomorphic to $F_\lambda^q(\Z_2)$ unless $\frac{1}{p}+\frac{1}{q}=1$.
%\emph{(Hint: for the upper bound, use the Riesz-Thorin interpolation theorem.)}
%\end{ex}

One can verify with elementary methods that the algebras $F_\lambda^p(\Z_2)$,
for different values of $p$, are pairwise not isometrically
isomorphic (unless $\frac{1}{p}+\frac{1}{q}=1$); 
see Proposition~3.2 in~\cite{GarThi_representations_2019}.
Moreover, this can be used to deduce a similar result for $F_\lambda^p(\Z)$;
see Theorem~3.5 in~\cite{GarThi_representations_2019}.

\subsection{Subgroups and quotients}
In this short subsection, we study some elementary functoriality properties
of $L^p$-operator group algebras.

\begin{rem}\label{ex:OpenSbgp}
Let $G$ be a locally compact group, let $H\subseteq G$ be an open subgroup,
and let $p\in [1,\I)$. Regard $L^1(H)$ as a subalgebra of $L^1(G)$
canonically (by extending a function on $H$ as zero on its complement). 
Then $\lambda_p^G|_{L^1(H)}$ is isometrically conjugate to 
the representation
\[\lambda_p^H\otimes \id_{\ell^p(G/H)}\colon L^1(H)\to 
\B\left(L^p(H)\otimes \ell^p(G/H)\right). 
\]
\end{rem}

%\begin{ex}
%Prove the claim in \autoref{ex:OpenSbgp} in the case that $G$ is countable and discrete.
%\end{ex}

\begin{prop}\label{prop:Sbgps}
Let $G$ be a locally compact group, let $H\subseteq G$ be an open subgroup,
and let $p\in [1,\I)$.
Denote by $\iota\colon L^1(H)\to L^1(G)$ the canonical isometric inclusion described
in \autoref{ex:OpenSbgp}. 
Then there are 
canonical injective contractive homomorphisms
\[\iota_\lambda^p\colon F^p_\lambda(H)\to F^p_\lambda(G) \ \ \mbox{ and } \ \
\iota^p\colon F^p(H)\to F^p(G).
\]
Moreover, $\iota_\lambda^p$ is isometric.
\end{prop}
\begin{proof}
For the map $\iota^p$, one needs to show that for all $f\in L^1(H)$
one has 
\[\|\iota(f)\|_{F^p(G)}\leq \|f\|_{F^p(H)}.\]
Given a contractive representation $\varphi\colon L^1(G)\to \B(E)$ on an 
$L^p$-space $E$, the composition $\psi=\varphi\circ\iota\colon L^1(H)\to \B(E)$
is also contractive, and one clearly has 
\[\|\varphi(\iota(f))\|= \|\psi(f)\|.\]
One readily checks, using the definition of $\|\cdot\|_{F^p(G)}$ as a supremum, 
that the above implies the desired inequality.

The result for $\iota_\lambda^p$ follows immediately from
\autoref{ex:OpenSbgp}, together with the fact that 
$\|\lambda_p^H(f)\otimes \id_{E}\|=\|\lambda^H_p(f)\|$ for all 
$f\in L^1(H)$ and all $L^p$-spaces $E$.
\end{proof}

For the following result, we will use, without proof, 
that if $N$ is a closed normal subgroup in a locally compact group $G$, 
then there exists a quotient map $\pi \colon L^1(G)\to L^1(G/N)$.

\begin{prop}\label{prop:Quotients}
Let $G$ be a locally compact group, let $N\subseteq G$ be a closed normal subgroup,
and let $p\in [1,\I)$. 
\be\item There is 
a canonical contractive homomorphism with dense range
\[\pi^p\colon F^p(G)\to F^p(G/N).\]
\item When $p>1$, there is 
a canonical contractive homomorphism with dense range
\[\pi_\lambda^p\colon F^p_\lambda(G)\to F^p_\lambda(G/N)\]
if and only if $N$ is amenable.
\ee
\end{prop}

Part~(1) is proved easily using the map $\pi\colon L^1(G)\to L^1(G/N)$
described before the proposition, and we leave its verification to the 
reader.
The proof of part~(2) (for $p>1$) uses the theory of weak
containment of representations, and we omit it. 
The map in~(2) exists for $p=1$ regardless of whether $N$ is amenable
or not: indeed, this is just the map $\pi\colon L^1(G)\to L^1(G/N)$. 

\subsection{The effect of changing the exponent $p$}
In this subsection, we look at the question of whether the algebras 
$F^p_\lambda(G)$ (or $F^p(G)$), for different values of $p$, are isometrically
isomorphic or anti-isomorphic. 

\begin{df}
Let $A$ be a Banach algebra. We define its \emph{opposite algebra} as the Banach
algebra $A^{\mathrm{opp}}$ whose underlying Banach space structure agrees with 
that of $A$, and where $a\cdot_{\mathrm{opp}} b=ba$ for all $a,b\in A$.
A representation of $A^{\mathrm{opp}}$ is naturally identified with an
\emph{anti-representation} of $A$ (namely one which is multiplicative with 
respect to the opposite multiplication).
\end{df}

In the following lemma, for a Banach space $E$ we denote by $E'$ its
dual space, and for an operator $a\in \B(E)$ we denote by $a'\in \B(E')$
its adjoint operator. Recall that $\|a'\|=\|a\|$ and $(ab)'=b'a'$
for all $a,b\in \B(E)$.

\begin{lma}
Given $p\in (1,\I)$, we denote by $p'\in (1,\I)$ its conjugate exponent. 
A algebra $A$ is an $L^p$-operator algebra if and only if $A^{\mathrm{opp}}$ is an 
$L^{p'}$-operator algebra.
\end{lma}
\begin{proof}
Fix an isometric representation $\varphi\colon A\to \B(E)$ on some 
$L^p$-space $E$. Define a linear map $\varphi'\colon A\to \B(E')$
by $\varphi'(a)=\varphi(a)'$ for all $a\in A$. Then $\varphi'$ is
isometric and anti-multiplicative. In other words, $\varphi'$ is an
isometric representation of $A^{\mathrm{opp}}$ on the $L^{p'}$-space
$E'$.
\end{proof}

Let $G$ be a locally compact group, and denote by $\Delta\colon G\to \mathbb{R}$
its modular function.
For $f\in L^1(G)$, let $f^\sharp\colon G\to\C$ be given by $f^\sharp(s) = \Delta(s^{-1}) f(s^{-1})$ for all
$s\in G$. 

\begin{prop}\label{ex:DualRegRep} 
Let $G$ be a locally compact group.
\be
\item Given $f\in L^1(G)$, the function $f^\sharp$ also belongs to $L^1(G)$. \item The resulting map $\sharp\colon L^1(G)\to L^1(G)$ is an anti-multiplicative isometric linear map of order two. 
\item Let $p\in (1,\I)$.
Then $\lambda_p(f)'=\lambda_{p'}(f^\sharp)$ for all $f\in L^1(G)$.
\ee
\end{prop}
\begin{proof}
(1). For $f\in L^1(G)$, we have
\[\|f^\sharp\|_1= \int_G \Delta(s^{-1})|f(s^{-1})| \ ds=\int_G f(s)\ ds=\|f\|_1, 
\]
as desired.

(2). The map $\sharp$ is clearly linear and isometric. Moreover, given
$f,g\in L^1(G)$ and $s\in G$, we have 
\begin{align*}
(f\ast g)^\sharp (s) &= \Delta(s^{-1}) (f\ast g)(s^{-1})\\
&= \Delta(s^{-1}) \int_G f(t)g(t^{-1}s^{-1}) \ dt\\
&= \Delta(s^{-1})\int_G f(s^{-1}t)g(t^{-1}) \ dt \\
&= \int_G \Delta(t^{-1}) g(t^{-1})\Delta(s^{-1}t)f(s^{-1}t) \ ds\\
&= (g^\sharp \ast f^\sharp)(s). 
\end{align*}
It follows that $\sharp$ is anti-multiplicative.

(3). Let $f\in L^1(G)$, let $\xi\in L^p(G)$, and let 
$\eta\in L^{p'}(G)$. Then
\begin{align*}
\langle \lambda_p(f^\sharp)(\xi),\eta\rangle 
&=\langle f^\sharp\ast \xi, \eta \rangle \\
&= \int_G (f^\sharp\ast \xi)(s) \eta(s) ds \\
&= \int_G \left( \int_G \Delta(t^{-1}) f^\sharp(st^{-1})\xi(t) dt \right) \eta(s) ds \\
&= \int_G \int_G \Delta(t^{-1}) \Delta(ts^{-1}) f(ts^{-1}) \xi(t) \eta(s) dt ds \\
&= \int_G \left( \int_G \Delta(s^{-1}) f(ts^{-1}) \eta(s) ds \right) \xi(t) dt \\
&= \int_G (f\ast \eta)(t) \xi(t) dt \\
&= \langle \xi, f\ast \eta \rangle\\
&= \langle \xi, \lambda_{p'}(f)(\eta)\rangle.
\end{align*}

It follows that $\lambda_p(f^\sharp)'=\lambda_{p'}(f)$,
as desired.
\end{proof}

\begin{prop}\label{prop:duality}
Let $G$ be a locally compact group, and let $p\in (1,\I)$. 
Then $\sharp\colon L^1(G)\to L^1(G)$ extends to isometric
anti-isomorphisms
\[F^p(G)\cong F^{p'}(G) \ \mbox{ and } \ F^p_\lambda(G)\cong F^{p'}_\lambda(G). \]
\end{prop}
\begin{proof}
We prove it for $F^p(G)$. Let $\pi\colon L^1(G)\to \B(E)$ be a contractive representation
on an $L^p$-space $E$. Denote by $\pi'\colon L^1(G)\to \B(E')$ the linear map given by
$\pi'(f)=\pi(f)'$ for all $f\in L^1(G)$. Then $E'$ is an $L^{p'}$-space, and the map $\pi'$ is contractive
(since an operator and its adjoint have the same norm) and anti-multiplicative.
Hence $\widetilde{\pi}=\pi'\circ\sharp\colon L^1(G)\to \B(E')$ is a contractive representation
satisfying $\|\widetilde{\pi}(f^\sharp)\|=\|\pi(f)\|$ for all $f\in L^1(G)$. Since the norm on $F^p(G)$
is universal with respect to contractive representations of $L^1(G)$ on $L^p$-spaces, it follows that $\sharp$
extends to an isometric anti-isomorphism $F^p(G)\cong F^{p'}(G)$. 

The claim for $F^p_\lambda(G)$ follows immediately from 
part~(2) of \autoref{ex:DualRegRep}. \end{proof}

Since $L^1(G)$ is self-opposite (via the map $\sharp$), it is tempting to guess that the universal
completion $F^p(G)$ is self-opposite. This is however not the case in general, as we explain next.
For a Banach algebra $A$, denote
by $\overline{A}^p$ the universal completion of $A$ with respect to all contractive
representations of $A$ on $L^p$-spaces. (For example, $F^p(G)=\overline{L^1(G)}^p$.)
It is tempting to claim that $\overline{A^{\mathrm{opp}}}^p$ is canonically the
opposite algebra of $\overline{A}^p$. (If this were true, what we said before about
$F^p(G)$ being self-opposite would follow.) Without further assumptions, this does
not seem to be true: the norm on the algebra $\overline{A^{\mathrm{opp}}}^p$ is
constructed using all \emph{anti-}representations of $A$ on $L^p$-spaces, while the norm on
$\overline{A}^p$ is defined using genuine representations. Since there is in general
no way to relate these two families of maps (given a representation of $A$ on an $L^p$-space, it is
not clear how to get an anti-representation of $A$ on some potentially different $L^p$-space), 
we do not see any relationship between
$\overline{A^{\mathrm{opp}}}^p$ and $(\overline{A}^p)^{\mathrm{opp}}$.

\begin{rem}
When $G$ is abelian, the anti-isomorphisms in \autoref{prop:duality} 
are trivially isomorphisms.
Since abelian groups are unimodular, the map $\sharp$ is just inversion on $G$, which in the
abelian case is an isomorphism. By composing again with the inversion, it follows that the identity
on $L^1(G)$ extends to
isometric isomorphisms between all the relevant completions for $p$ and $p'$. 
Except for $p=2$,
it is unclear whether there are any nonabelian groups for which the identity on $L^1(G)$ extends
to an isometric isomorphism $F^p_\lambda(G)\to F^{p'}_\lambda(G)$.
In fact, in his PhD thesis, Herz conjectured that this is never
the case. While the conjecture remains open in general, it has been confirmed for several classes
of groups.
\end{rem}

In view of \autoref{prop:duality}, one can restrict the attention to 
group algebras $F^p(G)$ and $F^p_\lambda(G)$ for 
H\"older exponents $p$ in $[1,2]$. The remaining question is whether the algebras
one gets for different values in $[1,2]$ are really different. This is indeed
the case:

\begin{thm}
Let $G$ be a nontrivial locally compact group, and 
let $p,q\in [1.2]$. Then the following
are equivalent:
\be
\item There is an isometric isomorphism $F^p(G)\cong F^q(G)$;
\item There is an isometric isomorphism $F_\lambda^p(G)\cong F_\lambda^q(G)$;
\item $p=q$.
\ee
\end{thm}

The theorem above is not just saying that the norms $\|\cdot\|_{F^p_\lambda(G)}$
and $\|\cdot\|_{F^q_\lambda(G)}$ (or $\|\cdot\|_{F^p(G)}$
and $\|\cdot\|_{F^q(G)}$) on $L^1(G)$ are different: it states that there are no
abstract isometric isomorphisms between their completions.

\section{Homomorphisms between convolution algebras}

In this section, 
we aim at describing all contractive, unital homomorphisms
between two $L^p$-operator group algebras. In particular, we want 
to describe all isomorphisms between them. For $p=2$, this is very complicated, 
and we illustrate this through some examples.

\begin{eg}
The groups $\Z_4$ and $\Z_2\times \Z_2$ have the same group \ca, namely $\C^4$.
\end{eg}

\begin{eg}
The group von Neumann algebra of $\Z^n$, for $n=1,\ldots,\I$, is
$L^\I([0,1])$, independently of $n$.
\end{eg}

The following is one of the most important open problems in operator algebras, 
and is known as the ``free factor problem'':

\begin{pbm}
Is there an isomorphism $W^*(\mathbb{F}_2)\cong W^*(\mathbb{F}_3)$?
\end{pbm}

A positive answer to the above problem implies that 
$W^*(\mathbb{F}_2)\cong W^*(\mathbb{F}_n)$ for all $n\in\N$ with $n\geq 2$. 

We will see that, for $p\neq 2$, we can obtain a very satisfactory
description of the homomorphisms between group algebras, which in particular
implies that groups with isomorphic $L^p$-group algebras must themselves be
isomorphic; see \autoref{thm:StructureHomsGpAlgs}. In this section, we
will work exclusively with $F^p_\lambda$, although the results also hold for 
$PM_p(G)$ and $CV_p(G)$. The situation for $F^p(G)$ is unknown.

We begin with some preparatory results. 

\begin{rem}\label{prop:LeftRightCommute}
Let $G$ be a locally compact group and let $p\in [1,\I)$. For $s\in G$, let 
$\texttt{Rt}_s \in \mathrm{Isom}(L^p(G))$ be the invertible isometry given by 
\[\texttt{Rt}_s(\xi)(t)=\xi(ts)\]
for all $\xi\in L^p(G)$ and all $t\in G$. Then 
$\lambda_p(f)\circ \texttt{Rt}_s=\texttt{Rt}_s\circ\lambda_p(f)$ for all $f\in L^1(G)$ and 
all $s\in G$. In particular, every element of $F^p_\lambda(G)$, $PM_p(G)$ or 
$CV_p(G)$ commutes with $\texttt{Rt}_s$.
\end{rem}

\begin{lma}\label{lma:FplambdaGenInvIsom}
Let $G$ be a discrete group and let $p\in [1,\I)$. Define 
$\texttt{Lt} \colon G\to \mathrm{Isom}(\ell^p(G))$ by 
$\texttt{Lt}_s(\xi)(t)=\xi(s^{-1}t)$
for all $s,t\in G$ and all $\xi\in \ell^p(G)$. Then $F^p_\lambda(G)$
is the subalgebra of $\B(\ell^p(G))$ generated by $\{\texttt{Lt}_s\colon s\in G\}$.
\end{lma}
\begin{proof}
We have seen in \autoref{eg:LeftReg} that the integrated form of $\texttt{Lt}$ 
is $\lambda_p$. By part~(1) of \autoref{ex:ReprL1G}, 
the image of $\lambda_p$
is generated, as a Banach algebra, by the image of $\texttt{Lt}$, which is what
we wanted to show. 
\end{proof}

For a unital Banach algebra $A$, we write 
\[\mathrm{Isom}(A)=\{v\in A\colon v \mbox{ invertible and } \|v\|=\|v^{-1}\|=1\}.\]
Note that if $A$ is a unital subalgebra of another Banach algebra $B$, then 
$\mathrm{Isom}(A)$ is a subgroup of $\mathrm{Isom}(B)$.

\begin{thm}\label{thm:UnitaryGpFpG}
Let $G$ be a discrete group and let $p\in [1,\I)\setminus\{2\}$.
Then there is a natural identification of topological groups
\[\mathrm{Isom}(F^p_\lambda(G))\cong G\times \T,\]
where $\mathrm{Isom}(F^p_\lambda(G))$ is endowed with the norm topology, and $G\times\T$
is endowed with the product topology.
\end{thm}
\begin{proof}
Let $v\in \mathrm{Isom}(F^p_\lambda(G))$. Since $F^p_\lambda(G)$ is a unital subalgebra
of $\B(\ell^p(G))$, and since $p\neq 2$, by Lamperti's theorem 
\autoref{thm:Lamperti} (and \autoref{rem:liftTransf}) 
there exist a bijection $\varphi\colon G\to G$ and a measurable function
$h\colon G\to S^1$ such that
\[v(\xi)(s)=h(s) \xi(\varphi(s))\]
for all $\xi\in \ell^p(G)$ and all $s\in G$.
By \autoref{prop:LeftRightCommute}, we have 
$v\circ\rho_t=\rho_t\circ v$ for all $t\in G$. 
We evaluate on both sides of this identity:
\[v(\rho_t(\xi))(s)=h(s) (\rho_t(\xi)(\varphi(s)))=h(s) \xi(\varphi(s)t)\]
and 
\[(\rho_t\circ v)(\xi)(s)=v(\xi)(st)=h(st) \xi(\varphi(st)).\]
It follows that $h(s) \xi(\varphi(s)t)=h(st) \xi(\varphi(st))$ for 
all $s,t\in G$. When $s=1$, we get 
$h(1)\xi(\varphi(1)t)=h(t)\xi(\varphi(t))$ for all $t\in G$ and all $\xi\in \ell^p(G)$. Setting $\xi=\delta_r$ for some $r\in G$, we deduce that 
$\varphi(t)=\varphi(1)t$ and $h(1)=h(t)$ for all $t\in G$. 
With $g_v=\varphi(t)$ and $\alpha_v=h(1)$, this shows that 
$\varphi$ is left multiplication by $g_v$, and 
$h$ is the constant function $\alpha_v$. In other words,
\[v(\xi)(s)=\alpha_v \xi(g_vs)\]
for all $\xi\in \ell^p(G)$ and all $s\in G$.

Define $\theta\colon \mathrm{Isom}(F^p_\lambda(G))\to G\times \T$ by 
$\theta(v)=(g_v,\alpha_v)$ for all $v\in\mathrm{Isom}(F^p_\lambda(G))$. It is 
easy to check that $\theta$ is a group homomorphism, and it is clearly 
injective. Moreover, it is surjective by \autoref{lma:FplambdaGenInvIsom}.

The claim about the norm follows from the norm computation in part~(4) 
of~\autoref{prop:IsometriesLp}. 
\end{proof}

\begin{cor}\label{cor:RecoverGroup}
Let $G$ be a discrete group and let $p\in [1,\I)\setminus\{2\}$.
Then there exists a natural identification $G\cong \mathrm{Isom}(F^p_\lambda(G))/\sim_{h}$.
\end{cor}

The following is the structure theorem for maps between group algebras that we
were aiming at:

\begin{thm}\label{thm:StructureHomsGpAlgs}
Let $G$ and $H$ be discrete groups, let $p\in [1,\I)\setminus\{2\}$, and 
let $\varphi\colon F^p_\lambda(G)\to F^p_\lambda(H)$ be a unital, contractive
homomorphism. Then:
\be\item There exist group homomorphisms $\theta\colon G\to H$ 
and $\gamma\colon G\to S^1$ such that
\[\varphi(\texttt{Lt}^G_g)=\gamma(g)\texttt{Lt}_{\theta(g)}^H\]
for all $g\in G$.
\item If $p>1$, then the kernel of $\theta$ is amenable. 
\item $\theta$ is injective if and only if $\varphi$ is injective if and only if 
$\varphi$ is isometric. 
\ee

In particular, there is an isometric isomorphism 
$F^p_\lambda(G)\to F^p_\lambda(H)$ if and only if $G\cong H$.
\end{thm}
\begin{proof}
(1). Given $g\in G$, the element $\varphi(\texttt{Lt}^G_g)$ is an
invertible isometry in $F^p_\lambda(H)$. By \autoref{thm:UnitaryGpFpG},
there exist $\gamma(g)\in S^1$ and $\theta(g)\in H$ satisfying the 
condition in the statement. Denote by $\gamma\colon G\to S^1$ and
$\theta\colon G\to H$ the resulting maps. In order to show that they
are group homomorphisms, let $g_1,g_2\in G$. Then
\begin{align*}
\gamma(g_1g_2)\texttt{Lt}^H_{\theta(g_1g_2)}&=\varphi(\texttt{Lt}^G_{g_1g_2})\\
&=\varphi(\texttt{Lt}^G_{g_1})\varphi(\texttt{Lt}^G_{g_2})\\
&=\gamma(g_1) \texttt{Lt}^H_{\theta(g_1)}\gamma(g_2)\texttt{Lt}^H_{\theta(g_2)}\\
&=\gamma(g_1)\gamma(g_2) \texttt{Lt}^H_{\theta(g_1)\theta(g_2)}.
\end{align*}
It follows that $\gamma(g_1g_2)=\gamma(g_1)\gamma(g_2)$, and similarly 
for $\theta$, as desired.

(2). This is a consequence of part~(2) of \autoref{prop:Quotients}.

(3). It is clear that if $\varphi$ is isometric, then it is injective, 
and in this case $\theta$ is injective. Conversely, assume that $\theta$
is injective. Then $\varphi$ can be written as the following composition
\[F^p_\lambda(G)\to F^p_\lambda(G)\to F^p_\lambda(H),\]
where the first map is determined by 
$\texttt{Lt}^G_{g}\mapsto \gamma(g)\texttt{Lt}^G_{g}$ for all $g\in G$,
and where the second map is the canonical one induced by the 
embedding $\theta\colon G\hookrightarrow H$. The first of these is easily
seen to be an isometric isomorphism, while the second one is isometric
by \autoref{prop:Sbgps}. This proves (3). 

The last assertion follows from \autoref{cor:RecoverGroup}.
\end{proof}

\begin{rem}
In the context of \autoref{thm:StructureHomsGpAlgs},
it can happen that $F^p_\lambda(G)$ is isomorphic, but not isometrically,
to $F^p_\lambda(H)$, even though $G$ and $H$ are not isomorphic. The 
smallest such example is obtained by taking $G=\Z_4$ and $H=\Z_2\times \Z_2$. 
Then $G$ and $H$ are clearly not isomorphic, but $F^p_\lambda(G)$ and 
$F^p_\lambda(H)$ are both isomorphic (although not isometrically, unless
$p=2$) to 
$\C^4$.
\end{rem}

\section{Spatial partial isometries and graph algebras}\label{sec:Cuntz}

We begin with a general definition.

\begin{df}\label{df:PartialIso}
Let $A$ be an algebra. We say that an element $s\in A$ is a 
\emph{partial isometry} if there exists $t\in A$ such that $st$ and $ts$
are idempotents. 
\end{df}

The prototypical example of a partial isometry in $\B(\Hi)$ is given by 
a surjective isometry between subspaces; in fact, these are precisely the partial
isometries of norm one. 

For some purposes in $L^p$-operator algebras, one needs to work with 
partial isometries that are in some sense ``spatially implemented'', similarly
to how invertible isometries are spatially implemented by Lamperti's theorem.
This motivates the following definition.

\begin{df}\label{df:SpatialPartialIso}
Let $(\mathcal{A},\mu)$ be a localizable measured algebra. Given $E\in\mathcal{A}$,
we set $\mathcal{A}_E=\{E\cap F\colon F\in\mathcal{A}\}$ and let $\mu_E$
denote the restriction of $\mu$ to $\mathcal{A}_E$. Then $(\mathcal{A}_E,\mu_E)$
is also localizable. Observe that $L^p(\mu)\cong L^p(\mu_E)\oplus L^p(\mu_{E^c})$.
 
Given $E,F\in\mathcal{A}$, given an isomorphism 
$\varphi\colon \mathcal{A}_E \to \mathcal{A}_F$ of Boolean algebras, and given
$f\in \U(L^\I(\mu_F))$, the formula 
\[s(\xi)=f(\xi\circ\varphi) 
\left(\frac{d(\mu_E\circ\varphi^{-1})}{d\mu_F}\right)^{1/p}\]
for all $\xi\in L^p(\mu_E)$, defines an isometric isomorphism 
$L^p(\mu_E)\to L^p(\mu_F)$,
which can be regarded as a contractive map $s\colon L^p(\mu)\to L^p(\mu)$ (vanishing
on $L^p(\mu_{E^c})$). We 
call this map the \emph{spatial partial isometry} 
associated to $(E,F,\varphi, f)$. 
\end{df}

Spatial partial isometries are partial isometries in the sense of
\autoref{df:PartialIso}. In fact, the element $t$ is uniquely determined
and is the spatial partial isometry associated to $(F,E,\varphi^{-1},
\overline{f}\circ\varphi^{-1})$. (For $p=2$, this is just the 
adjoint of $s$.) Moreover, $st$ is the multiplication operator by the 
characteristic function of $F$, and similarly $ts$ is the multiplication
operator by the characteristic function of $E$. The following observation
is immediate, and will be needed in the sequel.

\begin{rem}\label{rem:HermIdemp}
Let $e\in \B(L^p(\mu))$ be a spatial idempotent. Then there exists 
$E\in\mathcal{A}$ such that $e$ is the multiplication operator by the 
characteristic function of $E$.
\end{rem}

Spatiality for partial isometries is defined in terms of the underlying measured
algebra. However, for $p\neq 2$ and as long as the measured algebra is 
localizable, the notion is independent of the underlying algebra.

\begin{prop}
Let $(\mathcal{A},\mu)$ and $(\mathcal{B},\nu)$ be localizable measure algebras, 
let $p\in [1,\I)\setminus\{2\}$. Suppose that there exists an isometric
isomorphism $u\colon L^p(\mu)\to L^p(\nu)$, and 
define an isometric isomorphism 
$\theta\colon \B(L^p(\mu))\to \B(L^p(\nu))$ by $\theta(a)=u\circ a\circ u^{-1}$
for all $a\in \B(L^p(\mu))$. Then an operator $s\in \B(L^p(\mu))$ is a spatial
partial isometry if and only if $\theta(s)$ is a spatial partial isometry. Moreover,
$t\in \B(L^p(\mu))$ is the reverse of $s$ if and only if $\theta(t)$ is
the reverse of $\theta(s)$.
\end{prop}
\begin{proof}
We will use that \autoref{thm:Lamperti} 
is valid for 
isometric isomorphisms between \emph{different} $L^p$-spaces.
Applied to $u$, this gives the existence of $h\in \mathcal{U}(L^\I(\nu))$
and a Boolean algebra isomorphism $\varphi\colon \mathcal{A}\to \mathcal{B}$ such that
\[u(\xi)=h(\varphi\circ\xi) \left(\frac{d(\mu\circ\varphi^{-1})}{d\nu}\right)^{1/p} 
\]
for all $\xi\in L^p(\mu)$. 
If $s$ is the spatial partial isometry associated to the quadruple
$(E,F,\sigma, f)$, it is easy to check that $\theta(s)$ is the 
spatial partial isometry associated to
\[\left(\varphi(E),\varphi(F),\varphi\circ\sigma\circ\varphi^{-1},f (h\circ\varphi^{-1})\right).\]

The converse, as well as the assertions regarding the reverse of $s$, are
proved analogously.
\end{proof}

In view of the previous proposition, for $p\neq 2$, 
it makes sense to say that an operator $s$ on an 
$L^p$-space $E$ is spatial without fixing a presentation of $E$ as $L^p(\mu)$ for
some localizable measure $\mu$.

An idempotent is always a partial isometry (take $s=t$). An idempotent which
is additionally a spatial partial isometry is called a \emph{spatial idempotent}.

\begin{prop}
Let $p\in [1,\I)\setminus\{2\}$, let $E$ be an $L^p$-space, 
and let $s\in B(L^p(\mu))$ be a partial isometry. Then $s$ is a spatial partial isometry
if and only if it is contractive and there exists $t\in \B(E)$ contractive 
such that $ts$ and $st$ are spatial idempotents.
\end{prop}
\begin{proof}
If $s$ is a spatial partial isometry, then it is contractive and 
its reverse $t$ satisfies the statement.

Conversely, suppose that $s$ is contractive and that there exist 
$t\in \B(E)$ and $E,F\in\mathcal{A}$ such that $ts=m_{\mathbbm{1}_E}$ and $st=m_{\mathbbm{1}_F}$
are spatial idempotents (see \autoref{rem:HermIdemp}).
Note that 
\[L^p(\mu)\cong L^p(\mu_E)\oplus L^p(\mu_{E^c})\cong L^p(\mu_F)\oplus L^p(\mu_{F^c}).\]
Moreover, $s$ restricts to the zero map on $L^p(\mu_{E^c})$ and 
the restriction of $s$ to $L^p(\mu_E)$ is a contractive map
$L^p(\mu_E)\to L^p(\mu_F)$. Applying a similar reasoning to $t$, we 
deduce that $ts$ is the identity on $L^p(\mu_E)$ and $ts$ is the 
identity on $L^p(\mu_F)$. It follows that $s|_{L^p(\mu_E)}$ is an invertible isometry
$L^p(\mu_E)\to L^p(\mu_F)$. Applying Lamperti's theorem, we 
obtain a spatial realization of $s|_{L^p(\mu_E)}$ (as an invertible
isometry) which then
gives a spatial realization of $s$ as a spatial partial isometry.
\end{proof}

Not all contractive partial isometries on an $L^p$-space, for $p\neq 2$, 
are spatial 
(unlike the case of invertible isometries, by Lamperti's theorem).
For example, the contractive idempotent $\frac{1}{2}\begin{bmatrix}
    1 & 1 \\
    1 & 1
  \end{bmatrix}\in M_2^p$ is not spatial.
  
\subsection{Spatial representations of matrix algebras}
For $n\in\N$, we denote by $c_n$ the counting measure on $\{1,\ldots,n\}$.
Note that $L^p(c_n)=\ell^p_n$ for all $p\in [1,\I)$.
Spatial partial isometries can be used to characterize the canonical matrix norms:

\begin{prop}\label{prop:CharactMnSpatial}
Let $n\in\N$, let $p\in [1,\I)\setminus\{2\}$, 
let $E$ be an $L^p$-space, and let $\varphi\colon M^p_n\to\B(E)$ be a (not necessarily
contractive) unital representation. Then the following are equivalent:
\be
\item $\varphi$ is isometric;
\item $\|\varphi(e_{j,k})\|=1$ and $\varphi(e_{j,j})$ is a spatial idempotent;
\item $\varphi(e_{j,k})$ is a spatial partial isometry, for all $j,k=1,\ldots,n$;
\item $\varphi(e_{j,k})$ is a spatial partial isometry with reverse 
$\varphi(e_{k,j})$, for all $j,k=1,\ldots,n$;
\item There exist another $L^p$-space $F$ and an isometric isomorphism
\[u\colon F\otimes \ell^p_n\to E\]
such that $\varphi(a)(u(\xi\otimes \eta))=u(\xi\otimes a\eta)$ for all
$\xi\in F$ and all $\eta\in \ell^p_n$. In other words, $\varphi$
is (conjugate to) a dilation of the canonical representation of $M_n^p$.
\ee
\end{prop}
\begin{proof}
It is clear that (5) $\Leftrightarrow$ (4) $\Leftrightarrow$ (3) 
$\Leftrightarrow$ (2) and that (5) $\Leftrightarrow$ (1).
We show that (1) implies (2) and that (2) implies (5).

Assume (1). Then $\|\varphi(e_{j,k})\|=1$
for all $j,k=1,\ldots,n$
because $\|e_{j,k}\|=1$ in $M_n^p$. Given $z\in S^1$ and $j=1,\ldots,n$,
set $v_{j,z}=1-e_{j,j}+ze_{j,j}$. One easily checks that $\|v_{j,z}\|\leq 1$,
that $v_{j,z}$ is invertible and that its inverse is $v_{j,\overline{z}}$.
It follows that $v_{j,z}$ belongs to $\mathrm{Isom}(M_n^p)$. 
Since $\varphi$ is unital, $\varphi(v_{j,z})$ is an invertible isometry 
on $E$. Fix a localizable measure algebra $(\mathcal{A},\mu)$ 
such that $E\cong L^p(\mu)$.
By Lamperti's theorem, there exist $h_{j,z}\in \U(L^\I(\mu))$
and $\psi_{t,j}\in \Aut(\mathcal{A})$ such that 
\[\varphi(v_{j,z})=m_{h_{j,z}}u_{\psi_{j,z}}.\]

On the other hand, $v_{j,z}$ is homotopic to $v_{j,1}=1$, so 
all the automorphisms $\psi_{j,z}$ must be the identity. 
Set $f_j=h_{j,-1}$ for $j=1,\ldots,n$. Since $v_{j,-1}$ has order two,
we deduce that the range of $f_j$ is contained in $\{1,-1\}\subseteq S^1$. 
It follows that 
\[1-2\varphi(e_{j,j})=\varphi(v_{j,-1})=m_{f_{j}}\]
and hence $\varphi(e_{j,j})=\frac{1-m_{f_j}}{2}$ is the multiplication 
operator by the characteristic function of the set where $f_j$ equals $-1$.
This, by definition, is a spatial idempotent.

Assume (2). For $j=1,\ldots,n$, choose $E_j\in\mathcal{A}$ such that
$\varphi(e_{j,j})$ is the multiplication operator by the characteristic 
function of $E_j$. Since $\varphi$ is unital, 
$\bigsqcup\limits_{j=1}^n E_j$ must be the total space in $\mathcal{A}$.
It is easy to see that $\varphi(e_{j,k})$ restricts to an isometric 
isomorphism from $L^p(\mu_{E_k})$ to $L^p(\mu_{E_j})$. 
Set $F=L^p(\mu_{E_1})$. Identify $F\otimes \ell^p_n$ with the 
space of $n$-tuples $(\xi_1,\ldots,\xi_n)$ with $\xi_1,\ldots,\xi_n\in F$
with the norm given by 
\[\|(\xi_1,\ldots,\xi_n)\|_p^p=\|\xi_1\|_p^p+\ldots+ \|\xi_n\|_p^p.\]
Define $u\colon F\otimes\ell^p_n\to E$ by setting
\[u(\xi_1,\ldots,\xi_n)=\xi_1+\varphi(e_{2,1})(\xi_2)+\cdots+\varphi(e_{n,1})(\xi_n)\]
for all $(\xi_1,\ldots,\xi_n)\in F\otimes\ell^p_n$. (Observe that
$\xi_1=\varphi(e_{1,1})(\xi_1)$.)
Then $u$ is isometric because the summands $\varphi(e_{j,1})\xi_j$ are supported
on disjoint subsets. One easily checks that $u$ is bijective and that 
$\varphi(e_{j,k})=u(e_{j,k}\otimes 1)u^{-1}$ for $j,k=1,\ldots,n$. 
We omit the details.
\end{proof}

A representation satisfying the equivalent conditions above is 
called \emph{spatial}.
In particular, we obtain another way of defining the spatial norm on
$M_n$: one defines an (algebraic, unital) representation $\rho$ of 
$M_n$ on an $L^p$-space to be \emph{spatial} if $\rho(e_{j,k})$ is
a spatial partial isometry for all $j,k=1,\ldots,n$. The previous proposition
shows that the $L^p$-operator norm $\|\cdot\|$ on $M_n$ defined by 
$\|x\|=\|\rho(x)\|$, for a spatial representation $\rho$, is independent
of $\rho$, and is in fact the same norm from \autoref{eg:Mn}. We will see
that this idea can be used in other contexts too, specifically to define 
$L^p$-analogs of the
Cuntz algebras and, more generally, of graph algebras.

\subsection{Analogs of Cuntz algebras on $L^p$-spaces}
In this subsection, which is based on \cite{Phi_analogs_2012} (generalizing
the case $p=2$ from \cite{Cun_simple_1977}), we discuss the $L^p$-Cuntz algebras.

We begin by defining their algebraic skeleton: the Leavitt algebras.

\begin{df}\label{df:Leavitt}
Let $n\in\N$ with $n\geq 2$. We define the \emph{Leavitt algebra} $L_n$
to be the universal complex algebra generated by elements $s_1,\ldots,s_n,
t_1,\ldots,t_n$, satisfying
\[t_js_k=\delta_{j,k} \ \ \mbox{ and } \ \ \sum_{j=1}^n s_jt_j=1\]
for all $j,k=1,\ldots,n$.
\end{df}

Leavitt algebras are interesting for many reasons. 
They were discovered by Leavitt in his attempts to show that there is no
way of defining a notion of dimension for free modules over general rings
(they case of $\Z$ being well-known to work).
Indeed, he showed that $L_n$, as a free $L_n$-module, satisfies
\[L_n\ncong \oplus_{j=1}^m L_n \mbox{ for } 1<m<n\ \ \mbox{ and } \ \
L_n\cong \oplus_{j=1}^n L_n.
\]
In particular, $L_n$ does not have the so-called 
``Invariant Basis Number Property''.
  
Next, we define a distinguished class of representations of Leavitt algebras
on $L^p$-spaces.

\begin{df}
Let $n\in\N$ with $n\geq 2$ and let $p\in [1,\I)$. An algebraic
unital representation $\rho\colon L_n\to \B(E)$ on an $L^p$-space $E$
is said to be \emph{spatial} if $\rho(s_j)$ is a spatial partial isometry
with reverse $\rho(t_j)$ for all $j=1,\ldots,n$.
\end{df}

It is easy to see that spatial representations exist, and 
the following is 
probably the easiest example. For notational convenience, we describe the 
representation for $L_2$.

\begin{eg}\label{eg:SpatialRep}
Let $p\in [1,\I)$.
Define operators $s_1,s_2,t_1,t_2$ on $\ell^p(\Z)$ by
setting 
\[s_j(e_n)=e_{2n+j} \ \ \mbox{ and } \ \ 
t_j(e_n)=\begin{cases*}
      e_{\frac{n-j}{2}} & if $n-j$ is even \\
      0 & if $n-j$ is odd 
    \end{cases*} 
\]
for $j=1,2$ and $n\in\N$. Then 
$s_1$ and $s_2$ are spatial isometries with reverses
$t_1$ and $t_2$. Moreover, the universal property of $L_2$ implies that
there is a well-defined algebra homomorphism $\rho\colon L_2\to \B(\ell^p(\Z))$ satisfying $\rho(s_j)=s_j$ and $\rho(t_j)=t_j$ for $j=1,2$. 
\end{eg}

We now define the $L^p$-Cuntz algebras $\mathcal{O}_n^p$. 
For $p=2$, these are \ca s which
are called simply Cuntz algebras and denoted $\mathcal{O}_n$. They were 
introduced by Cuntz in the late 70's \cite{Cun_simple_1977} and play a 
fundamental role in the theory of simple \ca s. Their $L^p$-analogs are 
much more recent, and were introduced by Phillips in 2012 
\cite{Phi_analogs_2012}.

\begin{df}\label{df:LpCuntzAlg}
Let $n\in\N$ with $n\geq 2$ and let $p\in [1,\I)$. We define a spatial
$L^p$-operator norm on $L_n$ by setting
\[\|x\|=\sup\{\|\rho(x)\|: \rho\colon L_n\to \B(L^p(\mu)) \ \mbox{spatial representation}\}.\]
We define the \emph{$L^p$-operator Cuntz algebra} $\mathcal{O}_n^p$ to be 
the completion of $L_n$ in the above norm.
\end{df}

Besides the group algebras discussed in \autoref{df:GroupLpOpAlgs}, which
were introduced in the 70's, the $L^p$-operator Cuntz algebras were the first
class of examples of $L^p$-operator algebras that was considered. 
The motivation was the following: Corti\~nas and Phillips showed that for 
a class of \ca s that contains $\On$, topological $K$-theory 
$K_\ast=K_\ast^{\mathrm{top}}$ and algebraic $K$-theory $K_\ast^{\mathrm{alg}}$
agree naturally. They suspected that their methods were applicable to other
Banach algebras, and the search for such examples led Phillips to define the 
algebras $\mathcal{O}_n^p$. It seems to be still unknown whether the 
algebraic and topological $K$-theories of $\mathcal{O}_n^p$ agree:

\begin{qst}
Let $n\in\N$ with $n\geq 2$ and let $p\in [1,\I)$. Is there a natural
isomorphism $K_\ast(\mathcal{O}_n^p)\cong 
K_\ast^{\mathrm{alg}}(\mathcal{O}_n^p)$?
\end{qst}

As is the case for matrix algebras (\autoref{prop:CharactMnSpatial}),
any two spatial representations induce the same norm:

\begin{thm}
Let $n\in\N$ with $n\geq 2$ and let $p\in [1,\I)$. 
Then any two spatial $L^p$-representations of $L_n$ induce the same norm. 
\end{thm}

In particular, it follows from the previous theorem 
and \autoref{eg:SpatialRep} that $\mathcal{O}_2^p$ (and in fact 
any $\mathcal{O}_n^p$) can be isometrically represented on $\ell^p(\Z)$.
  
%\begin{ex}
%Find an isometric representation of $\mathcal{O}_2^p$ on $L^p([0,1])$.
%\end{ex}

The $L^p$-Cuntz algebras are remarkable algebras that satisfy a number of
very relevant properties:

\begin{thm}
Let $n\in\N$ with $n\geq 2$ and let $p\in [1,\I)$. Then
\be\item $\mathcal{O}_n^p$ is simple.
\item $\mathcal{O}_n^p$ is purely infinite: for all $x\in \mathcal{O}_n^p$ with 
$x\neq 0$, there exist $a,b\in \mathcal{O}_n^p$ with $axb=1$.
\item $K_0(\mathcal{O}_n^p)\cong \Z_{n-1}$ and $K_1(\mathcal{O}_n^p)\cong \{0\}$.
\ee
\end{thm}

None of these results are particularly easy to prove, and 
the first proofs for $p=2$ (which appeared much earlier) are very different 
from the case $p\neq 2$. 
Indeed, the method originally used by
Cuntz to compute $K_\ast(\mathcal{O}_n)$ 
breaks down for $p\neq 2$, since it used the fact that the group of 
unitary matrices in $M_n$ is connected 
(while the group of invertible isometries
in $M_n^p$ is not, by Lamperti's theorem). 

The argument which does carry
over to the case $p\in [1,\I)$ consists in expressing $\mathcal{O}_n^p$ as the 
crossed product of the spatial $L^p$-operator UHF-algebra of type $n^\I$
(this is essentially the infinite tensor product of copies of $M_n^p$)
by the shift automorphism. Once this is accomplished, there exists a 
6-term exact sequence (the ``Pimsner-Voiculescu exact sequence'') 
in $K$-theory relating the $K$-groups of an 
$L^p$-operator algebra $A$ and the $K$-groups of its crossed product by $\Z$.
We omit the details.

Another result about Cuntz \ca s which proved to be absolutely fundamental in
the theory of simple \ca s is the following theorem of Elliott; see
\cite{Ror_short_1994} for a published proof. 

\begin{thm}
There is an (isometric) isomorphism $\Ot\otimes \Ot\cong \Ot$. 
\end{thm}

This result motivated the search for analogs of this theorem in other contexts.
For Leavitt algebras, Ara and Corti\~nas showed that there is no such isomorphism;
see \cite{AraCor_tensor_2013}:

\begin{thm}
There is no isomorphism $L_2\otimes L_2\cong L_2$. 
\end{thm}

It remained open whether there is an isomorphism for the $L^p$-versions of 
$\Ot$. In joint work with Choi and Thiel (see~\cite{ChoGarThi_rigidity_2019}), 
we have shown that this is also not the case;
a proof will be outlined in the last section.
This shows that the $C^*$-case is really quite special and that there are many
more isomorphisms between \ca s than between $L^p$-operator algebras.

\subsection{$L^p$-operator algebras of finite directed graphs}
This subsection is based on \cite{CorRod_operator_2017}.
Let $Q$ be a finite directed graph, which we write as
$Q=Q^{(0)}\cup Q^{(1)}$, where the elements of $Q^{(0)}$ are the 
\emph{vertices} and the elements of $Q^{(1)}$ are the \emph{edges}. 
We denote by $d, r \colon Q^{(1)}\to Q^{(0)}$ the \emph{domain}
and \emph{range} maps. 

We begin by defining the Leavitt path algebra associated to a graph.
To avoid technicalities, we assume that for every $v\in Q^{(0)}$
there exists an edge $a\in Q^{(1)}$ such that $r(a)=v$. (In the 
standard terminology, this means that every vertex in $Q$ is regular.)

\begin{df}
Let $Q$ be a finite oriented graph. We define its associated \emph{Leavitt path 
algebra} $L_Q$ to be the universal unital complex algebra
generated by 
elements $e_v, s_a, t_b$, for $v\in Q^{(0)}$ and $a,b\in Q^{(1)}$, subject
to the following relations:
\be\item $e_ve_w=\delta_{v,w}e_v$ for $v,w\in Q^{(0)}$;
\item $e_{r(a)}s_a=s_ae_{d(a)}=s_a$ for all $a\in Q^{(1)}$;
\item $t_ae_{r(a)}=e_{d(a)}t_a=t_a$ for all $a\in Q^{(1)}$;
\item $t_as_b=e_{d(b)}\delta_{s,b}$ for all $a,b\in Q^{(1)}$;
\item $e_v=\sum\limits_{\{a\in Q^{(1)}\colon r(a)=v\}} s_at_a$ for all $v\in Q^{(0)}$.
\ee
\end{df}

$L^p$-operator graph algebras are defined similarly to how $L^p$-operator
Cuntz algebras were defined in \autoref{df:LpCuntzAlg}: one considers the 
completion of the Leavitt path algebra with respect to spatial representations.

\begin{df}
Let $Q$ be a finite oriented graph and let $p\in [1,\I)$. 
Given an $L^p$-space $E$ and a 
representation $\varphi\colon L_Q\to \B(E)$, we say that $\varphi$ is 
\emph{spatial} if 
\bi
\item $\varphi(e_v)$ is a spatial idempotent for all $v\in Q^{(1)}$; 
\item $\varphi(s_a)$ and $\varphi(t_a)$ are spatial partial isometries 
for all $a\in Q^{(1)}$.\ei
We define the associated \emph{$L^p$-operator graph algebra} 
$\mathcal{O}^p(Q)$ to be the completion of $L_Q$ in the norm
\[\|x\|=\sup\{\|\varphi(x)\|\colon \varphi \mbox{ spatial representation on an } 
L^p\mbox{-space}\}.\]
\end{df}

Unlike in the case for Cuntz algebras, it is not in general true that \emph{any}
two spatial representations of $L_Q$ induce the same norm. In other words, the 
supremum in the previous definition is actually necessary. 
This can already be seen in the following case.

\begin{eg}
Let $C$ denote the graph with one vertex and one loop around it. 
Then $\mathcal{O}^p(C)\cong F^p(\Z)$ for all $p\in [1,\I)$.
\end{eg}

\begin{eg}
Let $n\in\N$ with $n\geq 2$. Denote by $C_n$ the graph with 
one vertex and $n$ loops around it. Then 
$\mathcal{O}^p(C_n)\cong \mathcal{O}_n^p$ for all $p\in [1,\I)$.
\end{eg}

\begin{eg}
For $n\in\N$, let $Q_n$ be the following graph 
\begin{eqnarray*} \xymatrix{
\bullet_1 \ar[r] & \bullet_2 \ar[r] &\cdots \ar[r] & \bullet_n}\end{eqnarray*}
For $p\in [1,\I)$, there is an isometric isomorphism $\mathcal{O}^p(Q_n)\cong M_n^p$.
\end{eg}

There is so far not so much known about $L^p$-operator graph algebras, 
although graph $C^*$-algebras are a very well-studied class with a number of
very nice properties. A thorough and systematic study of $L^p$-operator graph
algebras is at this point within reach and certainly very interesting.
  
\section{Crossed products and their isomorphisms}
Crossed products are a generalization of group algebras, where one considers
an action of a group on a topological space or, more generally, an $L^p$-operator
algebra, and constructs an ``enveloping algebra'' that encodes dynamical information
of the original action.

\subsection{Construction of crossed products}
Let $G$ be a locally compact group, endowed with its Haar measure $\mu$, 
let $A$ be an $L^p$-operator algebra, and let $\alpha\colon G\to\Aut(A)$ be
an action (by isometric isomorphisms). For example, one could take $A=C_0(X)$
for $X$ locally compact and Hausdorff, take an action $G\curvearrowright X$ by
homeomorphisms, and for $s\in G$ let $\alpha_s\colon C_0(X)\to C_0(X)$ be given by
$\alpha_s(f)(x)=f(s^{-1}\cdot x)$ for $f\in C_0(X)$ and $x\in X$.

Our next goal is to define the (reduced) crossed product $F_\lambda^p(G,A,\alpha)$, also
denoted $A\rtimes_{\alpha,\lambda}G$ whenever $p$ is clear from the context. 
Intuitively speaking, and by analogy with the semidirect product of groups,
the crossed product $F^p_\lambda(G,A,\alpha)$ 
is the ``smallest'' algebra that contains $A$ and 
$G$, and where the action of $G$ on $A$ is implemented by conjugation by 
invertible isometries.

The reduced crossed product will be constructed as a certain completion of the Banach algebra 
$L^1(G,A,\alpha)$, which as a Banach space agrees with $L^1(G,A)$ and 
whose product is given by \emph{twisted convolution}:
\[(f\ast g)(s)=\int_G f(t)\alpha_t(g(s^{-1}t))\ d\mu(t)\]
for all $f,g\in L^1(G,A,\alpha)$. 

\begin{df}
A \emph{covariant representation} of $(G,A,\alpha)$ on an $L^p$-space $E$
is a pair $(\varphi,u)$ where $\varphi\colon A\to \B(E)$ is a representation
and $u\colon G\to \mathrm{Isom}(E)$ is a group homomorphism, which satisfy
\[u_s\varphi(a)u_s^{-1}=\varphi(\alpha_s(a))\]
for all $s\in G$ and all $a\in A$.

Given a covariant pair $(\varphi, u)$, we define the associated \emph{integrated
representation} $\varphi\rtimes u\colon L^1(G,A,\alpha)\to \B(E)$ by
\[(\varphi\rtimes u)(f)(\xi)=\int_G \varphi(f(s))(u_s(\xi)) \ d\mu(s)\]
for all $f\in L^1(G,A,\alpha)$ and all $\xi\in E$.
\end{df}

We will consider a distinguished class of covariant pairs, called the 
\emph{regular} covariant pairs.

\begin{df}
Let $\varphi_0\colon A\to\B(E_0)$ be any representation. 
Set $E=L^p(G,E_0)$, and define the associated \emph{regular covariant pair} 
$(\varphi,u)$ by
\[\varphi(a)(\xi)(s)=\varphi_0(\alpha_{s^{-1}}(a))(\xi(s)) \ \ \mbox{ and } \ \
 u_s(\xi)(t)=\xi(s^{-1}t)
\]
for all $a\in A$, all $s,t\in G$ and all $\xi\in E$. 
\end{df}

\begin{df}
The \emph{reduced crossed product} $F^p_\lambda(G,A,\alpha)$ 
is the completion of $L^1(G,A,\alpha)$ in the norm
\[\|f\|=\sup\{\|(\varphi\rtimes u)(f)\|\colon (\varphi,u) \mbox{ is a regular covariant pair}\}.\]
\end{df}

\begin{eg}
Let $G$ be a locally compact group and let $p\in [1,\I)$.
Then there is a canonical identification 
$F_\lambda^p(G,\C)\cong F^p_\lambda(G)$. This follows easily from the
fact that the integrated form of any regular covariant pair for 
$(G,\C, \mathrm{trivial})$ is an amplification of the left regular
representation of $G$, and hence induces the same norm. 
\end{eg}

When $A$ has the form $C(X)$, so that the action $\alpha$ comes from an action
of $G$ on $X$ via homeomorphisms, we usually write $F^p_\lambda(G,X)$
instead of $F^p_\lambda(G,C(X))$.

There is also a \emph{full crossed product} $F^p(G,A,\alpha)$ which is defined
using all covariant pairs and not just those that are regular. The 
full crossed product, being universal for all covariant pairs, admits a very
nice description in terms of generator and relations. Moreover, when $G$
is amenable, then $F^p(G,A,\alpha)$ and $F^p_\lambda(G,A,\alpha)$ agree canonically;
the case $A=\C$ is \autoref{thm:GpAmen}. As a consequence, whenever $G$ is amenable,
the reduced crossed product $F^p_\lambda(G,A,\alpha)$ can be described in a very
concrete way. A particularly nice case if that of integer actions:

\begin{thm}
Take $G=\Z$ and $X$ compact and Hausdorff. Then an action of $\Z$ on $C(X)$
is generated by one homeomorphism $h\colon X\to X$, and 
$F^p_\lambda(\Z,X)$ is the universal
$L^p$-operator algebra generated by a copy of $C(X)$, an invertible isometry $u$
and its inverse, subject to the relation 
\[ufu^{-1}=f\circ h^{-1}.\]
\end{thm}

\subsection{Isomorphisms of crossed products}
The study of crossed products, particularly of those of the form
$F^p_\lambda(G,X)$, is a very active area of research within 
operator algebras. In this setting,
one tries to understand what properties of the dynamics $G\curvearrowright X$
are reflected in the algebraic structure of the crossed product. 
In this section, which is based on \cite{ChoGarThi_rigidity_2019}, try to answer this question. 

\begin{eg}
Let $G$ be a finite group, acting on the compact Hausdorff space $X=G$
via left translation. Then $F^p_\lambda(G,G)\cong \B(\ell^p(G))$. 
In particular, the crossed product of $G\curvearrowright G$ \emph{only} 
remembers the cardinality of $G$. 
\end{eg}

Although it does not remember the group $G$, we will see that 
$F^p_\lambda(G,X)$ remembers $C(X)$ (and hence $X$), and more 
generally that it remembers the ``continuous orbit equivalence'' of the
action, at least when the action is essentially free\footnote{Recall that an action 
$G\curvearrowright X$ is
said to be \emph{essentially free} if for all $g\in G\setminus\{1\}$, the set
$\{x\in X\colon g\cdot x=x\}$ has empty interior.}.

\begin{df}\label{df:coe}
Let $G$ and $H$ be countable discrete groups, let $X$ and $Y$ be compact
Hausdorff spaces, and let $G\curvearrowright^\sigma X$ and $H\curvearrowright^\rho Y$ 
be actions. We say that $\sigma$ and $\rho$ are \emph{continuously orbit equivalent}, 
written $G\curvearrowright^\sigma X \sim_{\mathrm{coe}} H\curvearrowright^\rho Y$,
if there exist a homeomorphism $\theta\colon X\to Y$ and continuous maps
$c_H\colon G\times X\to H$ and $c_G\colon H\times Y\to G$ satisfying
\[\theta(\sigma_g(x))=\rho_{c_H(g,x)}(\theta(x)) \ \ \mbox{ and } \ \ 
 \theta^{-1}(\rho_h(y))=\sigma_{c_G(h,y)}(\theta^{-1}(y))
\]
for all $x\in X$, $y\in Y$, $g\in G$ and $h\in H$.
\end{df}

When two essentially free actions as above are continuously orbit equivalent, 
the maps $c_G$ and $c_H$ from the definition are
uniquely determined and satisfy certain cocycle conditions. These cocycle
conditions allow one to show that if two essentially free actions are continuously
orbit equivalent, then their reduced crossed products are naturally 
isometrically isomorphic, for all $p\in [1,\I)$. The following theorem asserts
that the converse is true for $p\neq 2$.

\begin{thm}\label{thm:RigidityDynSysts}
Let $p\in [1,\infty)\setminus\{2\}$, let $G$ and $H$ be countable discrete groups, 
let $X$ and $Y$ be compact
Hausdorff spaces, and let $G\curvearrowright X$ and $H\curvearrowright Y$ 
be topologically free actions. Then the following are equivalent:
\be
\item There is an isometric
isomorphism $F^p_\lambda(G,X)\cong F^p_\lambda(H,Y)$;
\item $G\curvearrowright X$ and $H\curvearrowright Y$ are continuously orbit equivalent.
\ee
\end{thm}

We will not prove \autoref{thm:RigidityDynSysts}, and we will only explain
how to prove that an isometric isomorphism 
$F^p_\lambda(G,X)\cong F^p_\lambda(H,Y)$ must map the canonical copy of 
$C(X)$ inside $F^p_\lambda(G,X)$ to the canonical copy of $C(Y)$ inside 
$F^p_\lambda(H,Y)$. This is attained using the notion of the \emph{$C^*$-core}
of an $L^p$-operator algebra.

\begin{thm}\label{thm:C*core}
Let $p\in [1,\infty)$, and let $A$ be a unital $L^p$-operator algebra.
Then there is a largest 
unital $C^*$-subalgebra $\mathrm{core}(A)$ of $A$,
called the \emph{$C^*$-core} of $A$. If $p\neq 2$, then $\mathrm{core}(A)$ is abelian, hence of the form $C(X_A)$ for a (uniquely determined) compact Hausdorff
space $X_A$.
\end{thm}

The proof that such an algebra exists is a bit technical for $p\neq 2$, but we give three
other ways of identifying it:
\bi\item Given any unital 
isometric representation $\varphi\colon A\to\B(L^p(\mu))$ of 
$A$ on an $L^p$-space $L^p(\mu)$, the set
\[A_{\mathrm{h}}=\{a\in A\colon \varphi(a)\in L^\I(\mu)_{\mathbb{\mathbb{R}}}\}\subseteq A\]
can be shown to be independent of $\varphi$, and is closed under
multiplication and pointwise complex conjugation (as a subset of $L^\I(\mu)$). 
In particular, the image of $A_{\mathrm{h}}+iA_{\mathrm{h}}$ 
under $\varphi$ is a norm-closed self-adjoint
subalgebra of $L^\I(\mu)$, and is therefore a commutative \ca. This is the $C^*$-core of $A$.
\item Given any unital 
isometric representation $\varphi\colon A\to\B(L^p(\mu))$ of 
$A$ on an $L^p$-space $L^p(\mu)$, the subgroup
\[\mathcal{V}(A)=\{u\in \mathrm{Isom}(A)\colon \varphi(u)\in L^\I(\mu)\}\subseteq
 \mathrm{Isom}(A)
\]
can be shown to be independent of $\varphi$. It is clearly commutative because
$L^\I(\mu)$ is commutative. Then $C(X_A)$ is the closed linear span of $\mathcal{V}(A)$.
\item Set 
\[\mathrm{Herm}(A)=\{a\in A\colon \|e^{ita}\|=1 \mbox{ for all } t\in\mathbb{R}\}.\]
Then $\mathrm{Herm}(A)+i\mathrm{Herm}(A)=C(X_A)$. (In fact, $\mathrm{Herm}(A)$ agrees
with the set $A_{\mathrm{h}}$ from the first bullet point above.)
\ei

\begin{eg}
Let $\mu$ be a localizable measure and let $p\neq 2$. Then
the core of $\B(L^p(\mu))$ is $L^\I(\mu)$. In particular, $X_{M_n^p}=\{1,\ldots,n\}$.
\end{eg}

The algebra $C(X_A)$ plays the role that \emph{maximal abelian subalgebras}
play in the context of \ca s, with two differences: it is unique (an advantage),
and it may be very small (a disadvantage). It may not even
be maximal abelian. For example:

\begin{eg}
Let $G$ be a discrete group and let $p\neq 2$. Then $X_{F^p_\lambda(G)}=\{\ast\}$. This follows from the second description of $C(X_A)$ given above: indeed, \autoref{thm:UnitaryGpFpG} shows that the only invertible isometries
in $F^p_\lambda(G)$ which are purely multiplication operators are the 
multiples of the unit.
\end{eg}

The following result clarifies how $C(X)$ is abstractly identified inside
$F^p_\lambda(G,X)$.

\begin{thm}
Let $G$ be a discrete group, let $X$ be a compact Hausdorff space, let 
$G\curvearrowright X$ be an action, and let $p\in [1,\I)\setminus\{2\}$.
Then the $C^*$-core of $F^p_\lambda(G,X)$ is $C(X)$.
\end{thm}

We close this section with some comments on what information about $G$ and 
$H$ one can deduce from knowing that they admit two continuously orbit 
equivalent actions. One of course does not expect to get an isomorphism of 
the groups, and in fact the type of equivalence one gets is really very weak
(although strong enough to give some interesting applications; see the following
section).

\begin{df}\label{df:qiGroups}
Let $G$ and $H$ be finitely generated groups, endowed with their word metrics $d_G$ and $d_H$. 
We say that $G$ and $H$ are
\emph{quasi-isometric} if there exist a function $\varphi\colon G\to H$ and a constant $K>0$ such that
\[K^{-1}d_G(g,g')-K \leq d_H(\varphi(g),\varphi(g'))\leq Kd_G(g,g')+K\]
for all $g,g'\in G$.
\end{df}

\begin{rem}\label{rem:COEqi}
By~Theorem~3.2 in~\cite{MedSauTho_cantor_2017},
if $G\curvearrowright X$ and $H\curvearrowright Y$ are continuously orbit equivalent, then $G$ is quasi-isometric to $H$. 
\end{rem}

\section{Tensor products of Cuntz algebras}
In this final section, also based on \cite{ChoGarThi_rigidity_2019}, 
we answer a question of Phillips regarding the existence of an isometric
isomorphism between $\mathcal{O}_2^p$ and 
$\mathcal{O}_2^p\otimes \mathcal{O}_2^p$
for $p\in [1,\I)\setminus\{2\}$. 

\begin{thm}\label{thm:IsomCtzAlgs}
Let $p\in [1,\infty)\setminus\{2\}$, let $n,m\in\N$.
Then there is an isometric isomorphism
\[\underbrace{\Otp\otimes_p\cdots\otimes_p\Otp}_\text{$n$} \cong \underbrace{\Otp\otimes_p\cdots\otimes_p\Otp}_\text{$m$} \]
if and only if $n=m$.
\end{thm}

The first step in proving the previous theorem is identifying 
$\Otp$ as a crossed product by an essentially free topological action.
This is done in the following proposition.

\begin{prop}\label{thm:CtzAlgsCrossedProd}
Let $p\in [1,\I)$. Then there exist an essentially free action
of $\Z_2\ast\Z_{3}$ on the Cantor set $X$ and an isometric isomorphism
\[F^p_\lambda(\Z_2\ast\Z_{3}, X)\cong \Otp.\]
\end{prop}
\begin{proof}
We only describe the action, and omit the construction of the isomorphism.
We identify the Cantor set $X$ as
\[X=\left\{x\colon \N\to \Z_2\ast\Z_{3} \ \middle\vert\ \ \begin{array}{@{}l@{}}
  \text{for $k\in \N$ there is $j_k\in \{2,3\}$ }\\
  \text{such that $x(k)\in \Z_{j_k}\subseteq \Z_2\ast\Z_{3}$}\\
  \text{and $j_k\neq j_{k+1}$ for all $k\in\N$}\\
  \end{array} \right\}.\]

We denote by $a\in \Z_2$ the nontrivial element, and by $b\in\Z_{3}$ the canonical generator of order $3$.
Define an action $\Z_2\ast\Z_{3}\curvearrowright X$ by
\begin{equation*}
(ax)(k)=\left\{\begin{array}{lr}
        x(k+1) & \text{if } x(0)=a;\\
        x(k) & \text{if } j_0=2, x(0)\neq a, \text{ and } k>0;\\
	ax(0) & \text{if } j_0=2, x(0)\neq a, \text{ and } k=0;\\
	x(k-1) & \text{if } j_0\neq 2, \text{ and } k>0;\\
	a & \text{if } j_0\neq 2, \text{ and } k=0,
        \end{array}\right.
\end{equation*}
and
\begin{equation*}
(bx)(k)=\left\{\begin{array}{lr}
        x(k+1) & \text{if } x(0)=b;\\
        x(k) & \text{if } j_0=2, x(0)\neq b^2, \text{ and } k>0;\\
	bx(0) & \text{if } j_0=2, x(0)\neq b^2, \text{ and } k=0;\\
	x(k-1) & \text{if } j_0\neq 2, \text{ and } k>0;\\
	b & \text{if } j_0\neq 2, \text{ and } k=0.
        \end{array}\right.
\end{equation*}
One checks that $a$ acts via a homeomorphism of order 2, and that $b$ acts via a
homeomorphism of order $3$, so that the previous equations really do define
an action of $\Z_2\ast\Z_3$ on $X$. We omit the details.
\end{proof}

It is not difficult to deduce from the previous proposition that
$\underbrace{\Otp\otimes_p\cdots\otimes_p\Otp}_\text{$n$}$ is isometrically
isomorphic to the crossed product of an essentially free 
action of $(\Z_2\ast\Z_3)^n$ on the 
Cantor space. We are now ready to finish the proof of \autoref{thm:IsomCtzAlgs}.
\newline

\emph{Proof of \autoref{thm:IsomCtzAlgs}.}
Suppose that there exists an isometric isomorphism
\[\underbrace{\Otp\otimes_p\cdots\otimes_p\Otp}_\text{$n$} \cong \underbrace{\Otp\otimes_p\cdots\otimes_p\Otp}_\text{$m$}.\]
By the comments above, there are actions of $(\Z_2\ast\Z_3)^n$
and $(\Z_2\ast\Z_3)^m$ on the Cantor space $X$ such that 
$F^p_\lambda((\Z_2\ast\Z_3)^n,X)\cong F^p_\lambda((\Z_2\ast\Z_3)^m,X)$.
By \autoref{thm:RigidityDynSysts}, this implies that the underlying
dynamical systems are continuously orbit equivalent. 
By \autoref{rem:COEqi}, this implies that 
$(\Z_2\ast\Z_3)^n$ is quasi-isometric to $(\Z_2\ast\Z_3)^m$. Finally, it can be shown, using asymptotic dimension, for groups that such a quasi-isometry exists if and only if $n=m$. This finishes
the proof.

%\bibliographystyle{siam}
%\bibliography{Ebibliography}

\begin{thebibliography}{10}

\bibitem{AraCor_tensor_2013}
{\sc P.~Ara and G.~Corti\~{n}as}, {\em Tensor products of {L}eavitt path
  algebras}, Proc. Amer. Math. Soc., 141 (2013), pp.~2629--2639.

\bibitem{BlePhi_operator_2018}
{\sc D.~Blecher and N.~C. Phillips}, {\em {${L^p}$ operator algebras with
  approximate identities I}},  (2018).
\newblock Preprint, arXiv:1802.04424.

\bibitem{ChoGarThi_rigidity_2019}
{\sc Y.~Choi, E.~Gardella, and H.~Thiel}, {\em {Rigidity results for
  {$L^p$}-operator algebras and applications}}, (2019).
\newblock Preprint, arXiv:1909.03612.

\bibitem{CorRod_operator_2017}
{\sc G.~Corti{\~{n}}as and M.~E. Rodr\'iguez}, {\em {$L^p$-operator algebras
  associated with oriented graphs}}, J. Operator Theory 81 (2019), pp.~225--254.

\bibitem{Cow_predual_1998}
{\sc M.~Cowling}, {\em The predual of the space of convolutors on a locally
  compact group}, Bull. Austral. Math. Soc., 57 (1998), pp.~409--414.

\bibitem{Cun_simple_1977}
{\sc J.~Cuntz}, {\em Simple {$C\sp*$}-algebras generated by isometries}, Comm.
  Math. Phys., 57 (1977), pp.~173--185.

\bibitem{DawSpr_convoluters_2019}
{\sc M.~Daws and N.~Spronk}, {\em On convoluters on {$L^p$}-spaces}, Studia
  Math., 245 (2019), pp.~15--31.

\bibitem{Der_property_2009}
{\sc A.~Derighetti}, {\em A property of {$B_p(G)$}. {A}pplications to
  convolution operators}, J. Funct. Anal., 256 (2009), pp.~928--939.

\bibitem{Der_book_2011}
\leavevmode\vrule height 2pt depth -1.6pt width 23pt, {\em Convolution
  operators on groups}, vol.~11 of Lecture Notes of the Unione Matematica
  Italiana, Springer, Heidelberg; UMI, Bologna, 2011.

\bibitem{DerFilMon_ideal_2004}
{\sc A.~Derighetti, M.~Filali, and M.~S. Monfared}, {\em On the ideal structure
  of some {B}anach algebras related to convolution operators on {$L^p(G)$}}, J.
  Funct. Anal., 215 (2004), pp.~341--365.

\bibitem{Fre_measure_2004}
{\sc D.~H. Fremlin}, {\em Measure theory. {V}ol. 3}, Torres Fremlin,
  Colchester, 2004.
\newblock Measure algebras, Corrected second printing of the 2002 original.

\bibitem{Gar_thesis_2015}
{\sc E.~Gardella}, {\em {Compact group actions on {C}*-algebras:
  classification, non-classifiability, and crossed products and rigidity
  results for $L^p$-operator algebras}}, ProQuest LLC, Ann Arbor, MI, 2015.
\newblock Thesis (Ph.D.)--University of Oregon.

\bibitem{GarLup_nonclassifiability_2016}
{\sc E.~Gardella and M.~Lupini}, {\em {Nonclassifiability of {UHF}
  {$L^p$}-operator algebras}}, Proc. Amer. Math. Soc., 144 (2016),
  pp.~2081--2091.

\bibitem{GarLup_representations_2017}
\leavevmode\vrule height 2pt depth -1.6pt width 23pt, {\em Representations of
  {\'e}tale groupoids on {$L^p$}-spaces}, Adv. Math., 318 (2017), pp.~233--278.

\bibitem{GarThi_banach_2015}
{\sc E.~Gardella and H.~Thiel}, {\em {Banach algebras generated by an
  invertible isometry of an {$L^p$}-space}}, J. Funct. Anal., 269 (2015),
  pp.~1796--1839.

\bibitem{GarThi_group_2015}
\leavevmode\vrule height 2pt depth -1.6pt width 23pt, {\em {Group algebras
  acting on {$L^p$}-spaces}}, J. Fourier Anal. Appl., 21 (2015),
  pp.~1310--1343.

\bibitem{GarThi_quotients_2016}
\leavevmode\vrule height 2pt depth -1.6pt width 23pt, {\em {Quotients of
  {B}anach algebras acting on {$L^p$}-spaces}}, Adv. Math., 296 (2016),
  pp.~85--92.

\bibitem{GarThi_extending_2017}
\leavevmode\vrule height 2pt depth -1.6pt width 23pt, {\em {Extending
  representations of Banach algebras to their biduals}}, Math. Z., to apper (2019).

\bibitem{GarThi_isomorphisms_2018}
\leavevmode\vrule height 2pt depth -1.6pt width 23pt, {\em {Isomorphisms of
  Algebras of Convolution Operators}},  (2018).
\newblock Preprint, arXiv:1809.01585.

\bibitem{GarThi_representations_2019}
\leavevmode\vrule height 2pt depth -1.6pt width 23pt, {\em Representations of
  {$p$}-convolution algebras on {$L^q$}-spaces}, Trans. Amer. Math. Soc., 371
  (2019), pp.~2207--2236.

\bibitem{Her_theory_1971}
{\sc C.~Herz}, {\em The theory of {$p$}-spaces with an application to
  convolution operators}, Trans. Amer. Math. Soc., 154 (1971), pp.~69--82.

\bibitem{Lam_isometries_1958}
{\sc J.~Lamperti}, {\em On the isometries of certain function-spaces}, Pacific
  J. Math., 8 (1958), pp.~459--466.

\bibitem{MedSauTho_cantor_2017}
{\sc K.~Medynets, R.~Sauer, and A.~Thom}, {\em Cantor systems and
  quasi-isometry of groups}, Bull. Lond. Math. Soc., 49 (2017), pp.~709--724.

\bibitem{NeuRun_column_2009}
{\sc M.~Neufang and V.~Runde}, {\em Column and row operator spaces over {${\rm
  QSL}_p$}-spaces and their use in abstract harmonic analysis}, J. Math. Anal.
  Appl., 349 (2009), pp.~21--29.

\bibitem{Phi_analogs_2012}
{\sc N.~C. Phillips}, {\em {Analogs of Cuntz algebras on $L^p$ spaces}},
  (2012).
\newblock Preprint, arXiv:1201.4196.

\bibitem{Phi_crossed_2013}
\leavevmode\vrule height 2pt depth -1.6pt width 23pt, {\em {Crossed products of
  $L^p$ operator algebras and the K-theory of Cuntz algebras on $L^p$ spaces}},
   (2013).
\newblock Preprint, arXiv:1309.6406.

\bibitem{PhiVio_classification_2017}
{\sc N.~C. Phillips and M.~G. Viola}, {\em {Classification of $L^p$ AF
  algebras}},  (2017).
\newblock Preprint, arXiv:1707.09257.

\bibitem{Ror_short_1994}
{\sc M.~R{\o}rdam}, {\em A short proof of {E}lliott's theorem:
  {${\mathcal{O}}_2\otimes{\mathcal{O}}_2\cong{\mathcal{O}}_2$}}, C. R. Math.
  Rep. Acad. Sci. Canada, 16 (1994), pp.~31--36.

\bibitem{Tza_remarks_1969}
{\sc L.~Tzafriri}, {\em Remarks on contractive projections in
  {$L_{p}$}-spaces}, Israel J. Math., 7 (1969), pp.~9--15.

\end{thebibliography}

\end{document}